\documentclass[reqno]{amsart}

\usepackage[numbers]{natbib}
\usepackage{amssymb, algorithm, caption, tikz, tikz-cd, zref-abspage, zref-user, tikzpagenodes, etoolbox, wrapfig, hyperref, array, enumitem, mathrsfs, color, colortbl, graphicx}
\usepackage[noend]{algpseudocode}
\usepackage[cal=boondox]{mathalfa}
\usepackage[abs]{overpic}

\usetikzlibrary{matrix,arrows,patterns,calc}
\hypersetup{colorlinks,linkcolor={blue},citecolor={blue},urlcolor={blue}}
\graphicspath{ {images/} }
\algrenewcommand\algorithmicindent{1.3em}

\definecolor{Gray}{gray}{0.9}
\definecolor{daisy}{RGB}{230,175,46}
\definecolor{pacific}{RGB}{69,78,158}

\newcommand{\Longtop}[2][.424\linewidth]{%
  \leavevmode\hfill\makebox[#1][l]{\begin{minipage}[t]{0.024\linewidth}$\triangleright$\end{minipage}\begin{minipage}[t]{0.4\linewidth}\setlength\parfillskip{0pt}#2\end{minipage}}}
\newcommand{\Top}[2][.424\linewidth]{%
  \leavevmode\hfill\makebox[#1][l]{\begin{minipage}[t]{0.024\linewidth}$\triangleright$\end{minipage}\begin{minipage}[t]{0.4\linewidth}#2\end{minipage}}}
\newcommand{\Bottom}[2][.424\linewidth]{%
  \leavevmode\hfill\makebox[#1][l]{\hspace{0.024\linewidth}#2}}
\algnewcommand\RETURN{\State \textbf{return} }

\makeatletter
\newcommand*\ALG@lastblockb{b}
\newcommand*\ALG@lastblocke{e}
\apptocmd{\ALG@beginblock}{%
    \ifx\ALG@lastblock\ALG@lastblockb
        \ifnum\theALG@nested>1\relax\expandafter\@firstoftwo\else\expandafter\@secondoftwo\fi{\ALG@tikzborder}{}%
    \fi
    \let\ALG@lastblock\ALG@lastblockb%
}{}{\errmessage{failed to patch}}

\pretocmd{\ALG@endblock}{%
    \ifx\ALG@lastblock\ALG@lastblocke
        \addtocounter{ALG@nested}{1}%
        \addtolength\ALG@tlm{\csname ALG@ind@\theALG@nested\endcsname}%
        \ifnum\theALG@nested>1\relax\expandafter\@firstoftwo\else\expandafter\@secondoftwo\fi{\endALG@tikzborder}{}%
        \addtolength\ALG@tlm{-\csname ALG@ind@\theALG@nested\endcsname}%
        \addtocounter{ALG@nested}{-1}%
    \fi
    \let\ALG@lastblock\ALG@lastblocke%
}{}{\errmessage{failed to patch}}

\tikzset{ALG@tikzborder/.style={line width=0.5pt,black}}

\newcommand*\currenttextarea{current page text area}

\newcommand*{\updatecurrenttextarea}{%
    \if@twocolumn
        \if@firstcolumn
            \renewcommand*{\currenttextarea}{current page column 1 area}%
        \else
            \renewcommand*{\currenttextarea}{current page column 2 area}%
        \fi
    \else
        \renewcommand*\currenttextarea{current page text area}%
    \fi
}

\newcounter{ALG@tikzborder}
\newcounter{ALG@totaltikzborder}
\newenvironment{ALG@tikzborder}[1][]{%
    \ifx&#1&\else
        \tikzset{ALG@tikzborder/.style={#1}}%
    \fi
    \stepcounter{ALG@totaltikzborder}%
    \expandafter\edef\csname ALG@ind@border@\theALG@nested\endcsname{\theALG@totaltikzborder}%
    \setcounter{ALG@tikzborder}{\csname ALG@ind@border@\theALG@nested\endcsname}%
    \tikz[overlay,remember picture] \coordinate (ALG@tikzborder-\theALG@tikzborder);% node {\theALG@tikzborder};% Modified \tikzmark macro
    \zlabel{ALG@tikzborder-begin-\theALG@tikzborder}%
    \ifnum\zref@extract{ALG@tikzborder-begin-\theALG@tikzborder}{abspage}=\zref@extract{ALG@tikzborder-end-\theALG@tikzborder}{abspage} \else
        \updatecurrenttextarea
        \ALG@drawvline{[shift={(0pt,.5\ht\strutbox)}]ALG@tikzborder-\theALG@tikzborder}{\currenttextarea.south east}{\ALG@thistlm}%
        \newcounter{ALG@tikzborderpages\theALG@tikzborder}%
        \setcounter{ALG@tikzborderpages\theALG@tikzborder}{\numexpr-\zref@extract{ALG@tikzborder-begin-\theALG@tikzborder}{abspage}+\zref@extract{ALG@tikzborder-end-\theALG@tikzborder}{abspage}}%
        \ifnum\value{ALG@tikzborderpages\theALG@tikzborder}>1
            \edef\nextcmd{\noexpand\AtBeginShipoutNext{\noexpand\ALG@tikzborderpage{\theALG@tikzborder}{\the\ALG@thistlm}}}%some pages need a border on the whole page
            \nextcmd
        \fi
    \fi
}{%
    \setcounter{ALG@tikzborder}{\csname ALG@ind@border@\theALG@nested\endcsname}%
    \tikz[overlay,remember picture] \coordinate (ALG@tikzborder-end-\theALG@tikzborder);% node {\theALG@tikzborder};% Modified \tikzmark macro
    \zlabel{ALG@tikzborder-end-\theALG@tikzborder}%
    \updatecurrenttextarea
    \ifnum\zref@extract{ALG@tikzborder-begin-\theALG@tikzborder}{abspage}=\zref@extract{ALG@tikzborder-end-\theALG@tikzborder}{abspage}\relax
        \ALG@drawvline{[shift={(0pt,.5\ht\strutbox)}]ALG@tikzborder-\theALG@tikzborder}{ALG@tikzborder-end-\theALG@tikzborder}{\ALG@thistlm}%
    \else
        \ALG@drawvline{\currenttextarea.north west}{ALG@tikzborder-end-\theALG@tikzborder}{\ALG@thistlm}%
    \fi
}

\newcommand*{\ALG@drawvline}[3]{%#1=from, #2=to, #3=value of \ALG@tlm/\ALG@thisthm
    \begin{tikzpicture}[overlay,remember picture]
        \draw [ALG@tikzborder]
            let \p0 = (\currenttextarea.north west), \p1=(#1), \p2 = (#2)
             in
            (#3+\fboxsep+.5\pgflinewidth+\x0+1.52em-\algorithmicindent,\y1+\fboxsep+.5\pgflinewidth)%-\fboxsep-.5\pgflinewidth
             --
            (#3+\fboxsep+.5\pgflinewidth+\x0+1.52em-\algorithmicindent,\y2-\fboxsep-.5\pgflinewidth)
        ;
    \end{tikzpicture}%
}

\newcommand{\ALG@tikzborderpage}[2]{%the whole page gets a border, #1=value of \theALG@tikzborder, #2=value of \ALG@tlm/\ALG@thistlm
    \updatecurrenttextarea
    \setcounter{ALG@tikzborder}{#1}%
    \ALG@drawvline{\currenttextarea.north west}{\currenttextarea.south east}{#2}%
    \addtocounter{ALG@tikzborderpages\theALG@tikzborder}{-1}%
    \ifnum\value{ALG@tikzborderpages\theALG@tikzborder}>1
        \AtBeginShipoutNext{\ALG@tikzborderpage{#1}{#2}}%
    \fi
    \vspace{-0.5\baselineskip}% Compensate for the generated extra space at begin of the page. No idea why exactly this happens.
}

\def\ALG@tikzbordertext{\the\ALG@tlm}
\makeatother
\makeatletter
\newlength{\ALG@continueindent}
\newcommand*{\ALG@customparshape}{\parshape 2 \leftmargin \linewidth \dimexpr\ALG@tlm+\ALG@continueindent\relax \dimexpr\linewidth+\leftmargin-\ALG@tlm-\ALG@continueindent\relax}
\newcommand*{\ALG@customparshapex}{\parshape 1 \dimexpr\ALG@tlm+\ALG@continueindent\relax \dimexpr\linewidth+\leftmargin-\ALG@tlm-\ALG@continueindent\relax}
\apptocmd{\ALG@beginblock}{\ALG@customparshape\everypar{\ALG@customparshapex}}{}{\errmessage{failed to patch}}
\makeatother
\newtheorem{theorem}{Theorem}[section]
\newtheorem{lemma}[theorem]{Lemma}
\newtheorem{proposition}[theorem]{Proposition}
\newtheorem{corollary}[theorem]{Corollary}

\newtheorem{innercustomthm}{Theorem}
\newenvironment{customthm}[1]
  {\renewcommand\theinnercustomthm{#1}\innercustomthm}
  {\endinnercustomthm}
  
\newtheorem{innercustomlem}{Lemma}
\newenvironment{customlem}[1]
  {\renewcommand\theinnercustomlem{#1}\innercustomlem}
  {\endinnercustomlem}

\theoremstyle{definition}
\newtheorem{definition}[theorem]{Definition}
\newtheorem{notation}[theorem]{Notation}

\numberwithin{equation}{section}

\newcommand{\bR}{\mathbb{R}}
\newcommand{\bQ}{\mathbb{Q}}
\newcommand{\bZ}{\mathbb{Z}}
\newcommand{\bN}{\mathbb{N}}
\newcommand{\bC}{\mathbb{C}}
\newcommand{\cO}{\mathcal{O}}
\newcommand{\fp}{\mathfrak{p}}
\newcommand{\fb}{\mathfrak{b}}
\newcommand{\eps}{\varepsilon}
\newcommand{\minus}{\raisebox{0.02cm}{-}\hspace{-0.01cm}}

\AtBeginDocument{
\addtolength{\abovedisplayskip}{0.8pt}
\addtolength{\abovedisplayshortskip}{0.8pt}
\addtolength{\belowdisplayskip}{0.8pt}
\addtolength{\belowdisplayshortskip}{0.8pt}}

\begin{document}

\title[Continued Fractions]{Continued Fractions Over Non-Euclidean Imaginary Quadratic Rings}

\thanks{}

\author{Daniel E. Martin}
\address{Department of Mathematics, Mathematical Sciences Building, One Shields Avenue, University of California, Davis, CA, 95616}
\email{dmartin@math.ucdavis.edu}

\subjclass[2010]{Primary: 11A55, 11J17, 11J70, 11Y65. Secondary: 11A05, 11R11, 11Y16, 11Y40, 40A15, 52C05.}

\keywords{continued fractions, Euclidean, Diophantine approximation, imaginary quadratic, nearest integer algorithm}

\thanks{This research was supported by NSF-CAREER CNS-1652238 under the supervision of PI Dr. Katherine E. Stange. The author is grateful to Kate Stange for so much guidance and help. The author also thanks the anonymous referee for corrections and helpful suggestions.}

\date{\today}

\begin{abstract}We propose and study a generalized continued fraction algorithm that can be executed in an arbitrary imaginary quadratic field, the novelty being a non-restriction to the five Euclidean cases. Many hallmark properties of classical continued fractions are shown to be retained, including exponential convergence, best-of-the-second-kind approximation quality (up to a constant), periodicity of quadratic irrational expansions, and polynomial time complexity.\end{abstract}

\maketitle

\section{Introduction}\label{sec:1}

Complex continued fractions were introduced by A. Hurwitz in 1887 \cite{hurwitz}, when he applied the nearest integer algorithm to $\bZ[i]$. His algorithm takes as input some $z = z_0\in\bC$ to be approximated. The $n^{\text{th}}$ coefficient, $a_n$, is then the nearest (Gaussian) integer to $z_{n-1}$. We stop if $a_n=z_{n-1}$, and continue with $z_n=1/(z_{n-1}-a_n)$ otherwise. The resulting approximations, called convergents, take the form $$\frac{p_n}{q_n} = a_1 +\cfrac{1}{a_2+\cfrac{1}{\raisebox{6pt}{$\ddots$} \;a_{n-1}+ \cfrac{1}{a_n}}}.$$ 

Hurwitz showed that many properties possessed by this algorithm over $\bZ$ still hold over $\bZ[i]$. For example, $|q_nz-p_n|$ decreases monotonically and exponentially, the continuants, denoted $q_n$ above, increase in magnitude monotonically and exponentially, and quadratic irrationals have periodic expansions. 

A key ingredient in his proofs is that $|z_{n-1}-a_n|$ is bounded by a constant less than $1$, namely $1/\sqrt{2}$. Such a constant exists precisely because open unit discs centered on lattice points of $\bZ[i]$ cover the complex plane. The same is true of the imaginary quadratic rings of discriminant $\Delta=-3$, $-7$, $-8$, and $-11$, but no others. This explains why the application and study of continued fractions over imaginary quadratic fields has been restricted to these five cases---the Euclidean ones. 

A large collection of references for Hurwitz' algorithm can be found in \cite{oswald} or \cite{robert2}. Also see \cite{lakein}, where Lakein investigates approximation quality of Hurwitz convergents in each Euclidean ring. See \cite{dani} for a similar algorithm removed from the ring setting, though still with a Euclidean-like requirement. See \cite{schmidt,schmidt2,schmidt3,schmidt11} for Schmidt's algorithm, which also only functions over the five Euclidean rings. Another approximation algorithm is given by Whitley in \cite{whitley}. It has continued fraction-like properties, while being executable in the four non-Euclidean, imaginary quadratic principal ideal domains, $\Delta=-19$, $-43$, $-67$, and $-163$. Whitley's idea was generalized to rings of class number $2$ by Bygott \cite{bygott}, and as he observes, it may be further adaptable to rings with trivial principal genus (the square of every ideal is principal).

Our purpose is to apply an algorithm with similar structure to that of Hurwitz in an arbitrary imaginary quadratic field.

\begin{notation}\label{not:field}Let $K$ be an imaginary quadratic field with ring of integers $\cO$ and discriminant $\Delta$.\end{notation} 

Our algorithm, Algorithm \ref{alg:1}, is presented in Subsection \ref{ss:alg} followed by an example execution when $\Delta=-23$. It does not build on the algorithm of Whitley and Bygott---the only setting in which the two coincide is a Euclidean ring, in which case both simply reduce to Hurwitz' algorithm.

Let us roughly summarize our way around the non-Euclidean obstacle. When there is no choice of coefficient $a_n\in\cO$ satisfying $|z_{n-1}-a_n|<1$, Algorithm \ref{alg:1} seeks $a_n$ near $b_nz_{n-1}$ instead, where $b_n$ comes from a fixed finite set $B\subset\cO\backslash\{0\}$. But the exact criteria for selecting $a_n$ and $b_n$ change according to the previous stage's choice of coefficient. We impose an analogue of the classical analytic restraint: \begin{equation}|b_nz_{n-1}-a_n|<|b_{n-1}|,\label{eq:100}\end{equation} and a new algebraic one: \begin{equation}b_{n-1}\,\Big|\,a_np_{n-1}+b_np_{n-2},\, a_nq_{n-1}+b_nq_{n-2}.\label{eq:101}\end{equation} The integer quotients from (\ref{eq:101}) are $p_n$ and $q_n$, and the algorithm continues with $z_n = b_{n-1}/(b_nz_{n-1}-a_{n-1})$. Remark that because $b_n$ need not equal $1$, our convergents are called \textit{generalized} continued fractions. (Some recent applications of generalized continued fractions over $\bZ$ can be found in \cite{anselm} and \cite{burger}.)

Among pairs $a_n\in\cO$ and $b_n\in B$ satisfying (\ref{eq:101}), at least one is guaranteed to satisfy (\ref{eq:100}) if open discs of center $a_n/b_n$ and radius $|b_{n-1}/b_n|$ cover $\bC$. If such a covering occurs for every $n$, we say $B$ is \textit{admissible} (defined more precisely in Definition \ref{def:admit}). The Hurwitz algorithm has a similar requirement: Euclideanity, which is equivalent to unit discs on integers covering $\bC$. These are the five rings for which $B=\{1\}$ is admissible.

For a given field there are many admissible sets, and each may give different continued fraction expansions of some input $z$. Even after fixing an admissible set, an input can have many possible continued fraction expansions because $z_{n-1}$ might lie in the overlap of multiple discs of center $a_n/b_n$ and radius $|b_{n-1}/b_n|$. Hurwitz deals with this situation by insisting that $a_n$ be nearest to $z_{n-1}$ (and $b_n=1$ always). Initially we make no such requirement to emphasize that the results of Section \ref{sec:3}, like the four following theorems, are valid independently of this choice. A method for selecting among many acceptable coefficients (Algorithm \ref{alg:2}) is not proposed until Section \ref{sec:4}. 

The first three results below are versions of the more general Theorems \ref{thm:main}, \ref{thm:best}, and \ref{thm:growq}, where constants (meaning with respect to $n$ and $z$) depend on $B$. For simplicity we have used $B=\big\{1,2,...,\big\lfloor\sqrt{|\Delta|}\big\rfloor\big\}$, which Theorem \ref{thm:ints} proves admissible, to get the following constants that depend only on $\Delta$.

\begin{theorem}\label{thm:1}If $n\geq 1$ then $|q_nz-p_n|$ is less than
$$i)\;\; \frac{2\sqrt{|\Delta|}}{|z_nq_n|},\hspace{1cm}ii)\;\;\frac{2\sqrt{|\Delta|}}{|q_{n+1}|},\hspace{0.5cm}\text{and}\hspace{0.5cm}iii)\;\;\frac{3|\Delta|}{|a_{n+1}q_n|}.$$\end{theorem}

\begin{theorem}\label{thm:2}If $p/q$ is not a convergent of $z$ for some $p,q\in\cO$ with $q\neq 0$, then $$|q_n(q_nz-p_n)|<16|\Delta q(qz-p)|$$ for any $n\geq 1$. That is, each $p_n/q_n$ is a best approximation of the second kind up to constants: if $rs\leq \sqrt{2}/16|\Delta|$, then $0<|q|<r|q_n|$ implies $|qz-p|>s|q_nz-p_n|$ for any $p\in\cO$ except perhaps when $p/q$ is already a convergent.\end{theorem}

\begin{theorem}\label{thm:3}If $\,0\leq n'<n$, then $16|\Delta q_n|> \sqrt{2}^{n-n'}|q_{n'}z_{n'}|$. In particular, if $n\geq 1$ then $16|\Delta q_n|>\sqrt{2}^n$.\end{theorem}

\begin{theorem}\label{thm:introq}There is a continued fraction expansion of $z$ in which the sequence of pairs $(a_n,b_n)_n$ is eventually periodic and infinite if and only if $[K(z):K]=2$.\end{theorem}

Note that the last statement refers to \textit{an} expansion rather than \textit{the} expansion due to the potential choice among coefficients that arises in the overlapping disc scenario. Figure \ref{fig:5} gives an example of how some expansions of a quadratic, irrational input can be periodic while others are not. A path in the right-side image can be periodic or aperiodic, depending on the choices made at those nodes which are the source of two arrows. Such a node corresponds to ``$z_{n-1}$" in the left-side image, which lies in the overlap of two discs, one for each arrow. More detail is given in Subsection \ref{ss:coef}.

Other results include the monotonic decrease of $|q_nz-p_n|$ (Proposition \ref{prop:mono}), an upper bound on $|qz-p|$ that implies $p/q$ appears as a convergent (Lemma \ref{lem:combo}), and equating bad approximability of $z$ to boundedness of $a_n/b_n$ (Corollary \ref{cor:bad}).

Variations of the properties above may hold for the algorithm of Whitley and Bygott in fields of class number $1$ or $2$, but approximation quality is not addressed in their work. Their goal was to compute spaces of cusp forms.

Section \ref{sec:4} shows that Algorithm \ref{alg:1} can be executed in any imaginary quadratic field by explicitly producing admissible sets in Theorem \ref{thm:ints}. The sets we give have two advantages over a generic one. The first is efficiency---the admissibilty requirement on $B$ guarantees coefficients exists, but not an easy way to find them. With $B$ as in Theorem \ref{thm:ints}, there is a subroutine for finding coefficients, Algorithm \ref{alg:2}, which gives Algorithm \ref{alg:1} polynomial complexity (Theorem \ref{thm:timebest}).

The second advantage to using $B$ from Theorem \ref{thm:ints} is control over $(p_n,q_n)$. In Euclidean rings with Hurwitz' algorithm or principal ideal domains with Whitley's, $(p_n,q_n)=\cO$. With Bygott's generalization to rings of class number $2$, all divisors of $(p_n,q_n)$ are ramified after appropriate scaling. A generic admissible set for Algorithm \ref{alg:1} loses such control, and thus potential applications like Whitley and Bygott's to the group $\text{PSL}_2(\cO)$. This can be partially remedied:

\begin{theorem}\label{thm:introcp}If $a_{n-1}$, $b_{n-1}$, $a_n$, and $b_n$ are found using Algorithm \ref{alg:2}, then only ramified, non-rational primes divide $(p_{n-1},q_{n-1},p_n,q_n)$.\end{theorem}

Admissible sets can also be precomputed for a particular ring. A brief explanation of how to do this is given in Subsection \ref{ss:precomp}. Sample output from the precomputation algorithm described is in Table \ref{table:2} for $|\Delta|<50$.

Some resources are available at \href{https://www.math.ucdavis.edu/~dmartin}{\url{math.ucdavis.edu/~dmartin}}, including the tool that created the images herein and C\texttt{++} source code for Algorithm \ref{alg:1} and for finding admissible sets. There is also software to create \textit{Schmidt arrangements} (coined and first studied by Stange \cite{stange}), fractal displays of circles in the complex plane obtained as the orbit of the real line under $\text{PSL}_2(\cO)$. It turns out that approximating $z\in\bC$ with Algorithm \ref{alg:1} corresponds to a ``walk" along circles in a Schmidt arrangement toward $z$. The convergents are exactly the points of intersection between successive circles in this walk. Details can be found in the author's dissertation \cite{martin}. Continued fractions are addressed on their own here for simplicity.

\section{A Continued Fraction Algorithm}\label{sec:2}

\subsection{Intuition for non-Euclidean rings}\label{ss:intuit}Hurwitz' algorithm can be applied in any imaginary quadratic ring, but with varying degrees of success. In this subsection we explore what happens if $\cO$ is not Euclidean through an example in $\bQ(\sqrt{-23})$. Recall notation from the first page, and let $M_0$ denote the identity matrix.

We will need the usual recursion relation $M_n = M_{n-1}S(a_n)$, where \begin{equation}M_n = \begin{bmatrix}p_n & p_{n-1} \\ q_n & q_{n-1}\end{bmatrix}\hspace{0.5cm}\text{and}\hspace{0.5cm}S(a) = \begin{bmatrix}a & 1 \\ 1 & 0\end{bmatrix}.\label{eq:2}\end{equation} With $z_n = 1/(z_{n-1}-a_n)$, it follows by induction that $z_n$ can be computed by applying the M\"{o}bius transformation associated with $M_n^{-1}$ to $z$. That is, \begin{equation}z_n = \frac{q_{n-1}z-p_{n-1}}{p_n-q_nz}.\label{eq:3}\end{equation} 

\begin{wrapfigure}{r}{0.41\textwidth}
  \centering
    \vspace{-0.3cm}
    \begin{overpic}[unit=1mm,trim=2cm 0cm 2cm 0cm,clip,height=5.8cm]{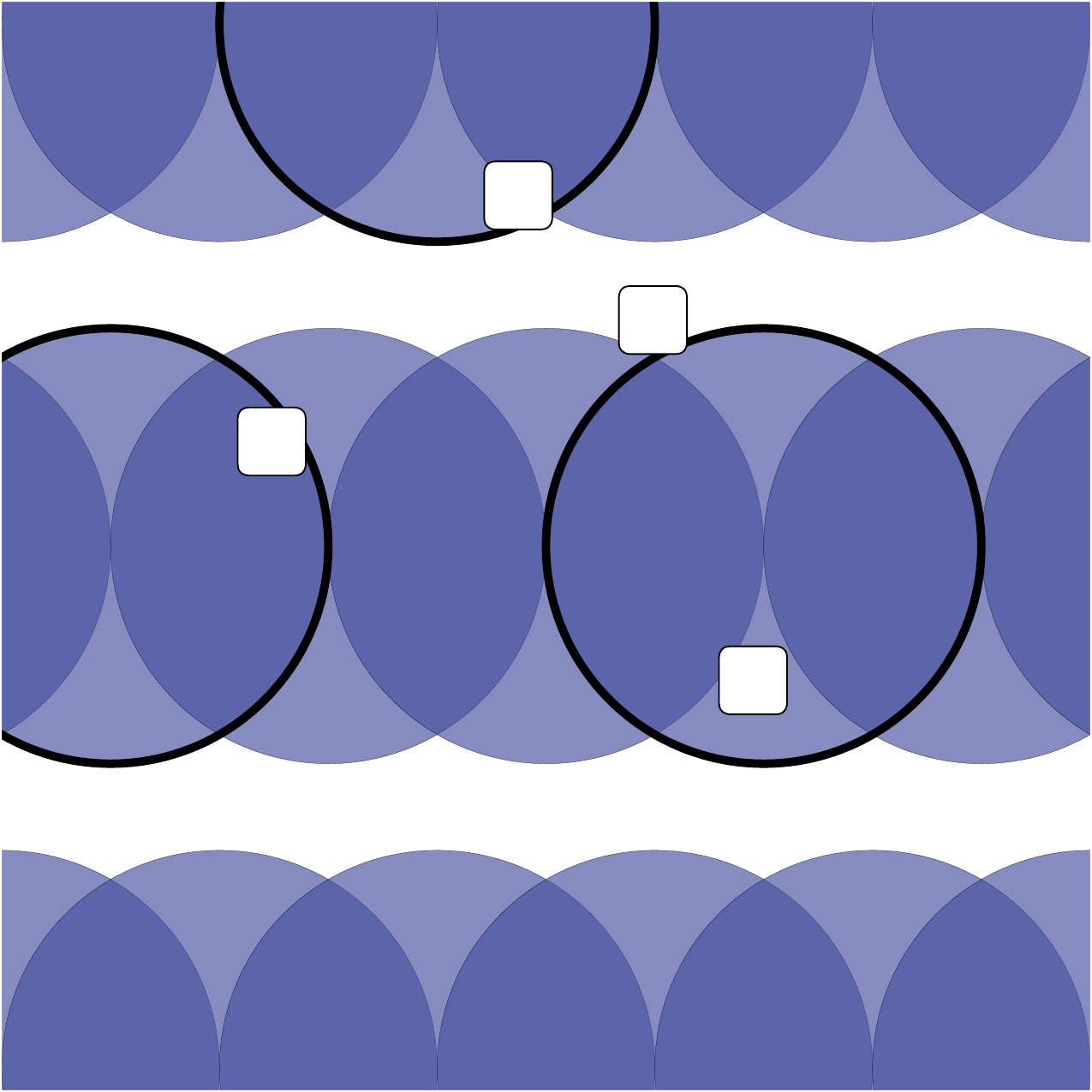}
    \put(8.3,33.5){\small$0$}
    \put(33.89,20.8){\small$1$}
    \put(21.35,46.55){\small$2$}
    \put(28.5,39.9){\small$3$}
    \end{overpic}
    \captionsetup{width=0.375\textwidth}
    \setlength\belowcaptionskip{-2\baselineskip}
  \caption{Unit discs around $z_0$, $z_1$, and $z_2$ with $|\Delta|=23$.}\label{fig:1}
\end{wrapfigure}
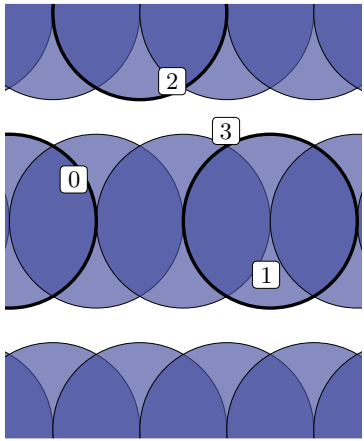

\noindent Thus an improvement in approximation quality, $|q_nz-p_n|<|q_{n-1}z-p_{n-1}|$, is equivalent to $1/|z_n|=|z_{n-1}-a_n| < 1$. So in a non-Euclidean ring, it is still desirable (and necessary, as we show shortly) that $z_{n-1}$ lie in the open unit disc centered on $a_n$.

Let us input $z=-1.26+0.48i$, labeled ``0" in Figure \ref{fig:1}, and take coefficients from the integers in $\bQ(\sqrt{-23})$. There are two choices for $a_1\in\cO$ whose unit discs contain $z_0$: $-1$ and $-2$. If $a_1 = -2$, for example, then

\vspace{0.15cm}

$\hfill\displaystyle z_1=\frac{1}{z_0-a_1}\approx 0.95-0.62i.\hfill$

\vspace{0.15cm}

\noindent Similarly, $a_2=1$ and $a_3 = (-1+\sqrt{-23})/2$ center the bold outlined unit discs that contain $z_1$ and $z_2\approx -0.13+1.61i$. But there is no such disc containing $z_3\approx 0.49+1.04i$. As a result, any choice of $a_4$ worsens approximation quality: $|q_4z-p_4|>|q_3z-p_3|$. 

We can persevere, perhaps searching for a clever combination $a_4,a_5,...,a_n$ to finally achieve $|q_nz-p_n|<|q_3z-p_3|$. Or at the very least, there may be a sequence of coefficients that makes $\lim_n p_n/q_n=z$. It happens that neither is possible. The obstruction is that $M_n$, up to a swapping of columns which we henceforth ignore, belongs to the elementary group in $\text{SL}_2(\cO)$---the group generated by $S(a)$ from (\ref{eq:2}) for $a\in\cO$. It is proved in \cite{martin2} that if $p$ and $q$ are the column entries of a matrix in the elementary group, then $p/q$ lies in the interior of a unit disc centered on an integer. Thus for any choices of $a_4,...,a_n$, the distance from $z_3$ to the column ratios of $M_3^{-1}M_n$, which belongs to the elementary group, is bounded from below by a positive constant. So the same is true of the distance between $z=M_3(z_3)$ and the column ratios and $M_n$. This is to say that no sequence of coefficients achieves $\lim_n p_n/q_n=z$.

A fix proposed by Whitley in \cite{whitley} is to permit right multiplication by certain additional matrices from $\text{SL}_2(\cO)$. So $M_n = M_{n-1}S$, where $S$ need not take the form $S(a)$. Generally, $|q_nz-p_n|<|q_{n-1}z-p_{n-1}|$ is equivalent to $|z-S_{1,1}/S_{2,1}| < 1/|S_{2,1}|$, thereby associating an open disc to $S$ which is no longer centered on an integer if $S_{2,1}$ is not a unit. Success occurs when we can choose matrices so that such discs cover $\bC$. This is possible exactly when $\cO$ is one of the eight principal ideal domains. In a non-principal ideal domain, there is a discrete set of problematic points. The so-called singular points are not covered by open discs with center $S_{1,1}/S_{2,1}$ and radius $1/|S_{2,1}|$ for $S\in\text{SL}_2(\cO)$ \cite{swan}. The approximation quality of Whitley's algorithm suffers when $(z_n)_n$ approaches a singular point.

Bygott goes a step further \cite{bygott} and allows $S$ from the extended Bianchi group (see Section 7.4 of \cite{grunewald} for a definition and basic properties). Only singular points $p/q$ for which $(p,q)^2$ is nonprincipal are left uncovered by the newly introduced open discs. Bygott works in fields of class number $2$ because no such points exist. 

To lengthen the list of imaginary quadratic fields that possess an approximation algorithm, we have gone from the elementary group to $\text{SL}_2(\cO)$ to the extended Bianchi group. There are no more extensions to attempt. The latter is maximal among discrete groups of M\"{o}bius transformations containing $\text{SL}_2(\cO)$ \cite{grunewald}. The group structure must be abandoned to obtain a covering of $\bC$ by open discs in fields with non-2-torsion ideal classes like $\bQ(\sqrt{-23})$. So let us return to the elementary group and consider the following modification to $S(a)$.

\begin{notation}\label{not:matrix}For $a,b\in\bC$ let $$S(a,b)=\begin{bmatrix}a & 1 \\ b & 0\end{bmatrix}.$$\end{notation}

It is well-known that open discs of radius $1/|b|$ and center $a/b$ cover $\bC$ for $a\in \cO$ and $b$ from some finite set $B\subset\cO\backslash\{0\}$. For example, $B=\{1,2\}$ works for $\bQ(\sqrt{-23})$, introducing discs of radius $1/2$ centered on half-integers. The resulting covering is the first image in Figure \ref{fig:3}. As shown in the second image, the closures of these discs still cover the plane after scaling radii by $\sqrt{8/9}$. Returning to our example, the first image shows $|z_3 - (1+\sqrt{-23})/4| < 1/2$. So $M_4=M_3S((1+\sqrt{-23})/2,2)$ gives $|q_4z-p_4|< |q_3z-p_3|$ as desired.

\begin{figure}
\centering
\begin{overpic}[unit=1mm,trim=2cm 0cm 2cm 0cm,clip,height=5.8cm]{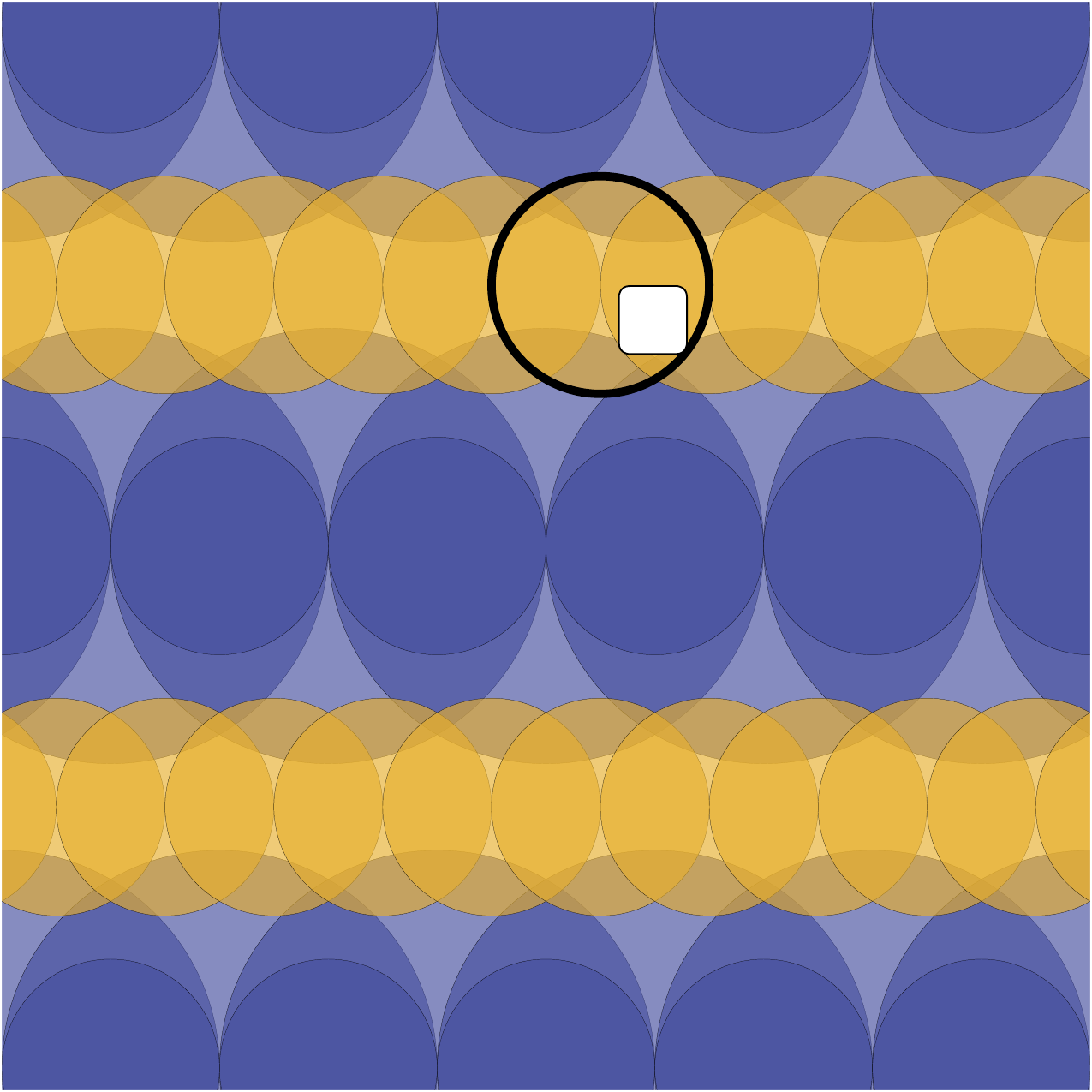}
    \put(28.5,39.9){\small$3$}
\end{overpic}\hspace{0.5cm}
\includegraphics[trim=2cm 0cm 2cm 0cm,clip,height=5.8cm]{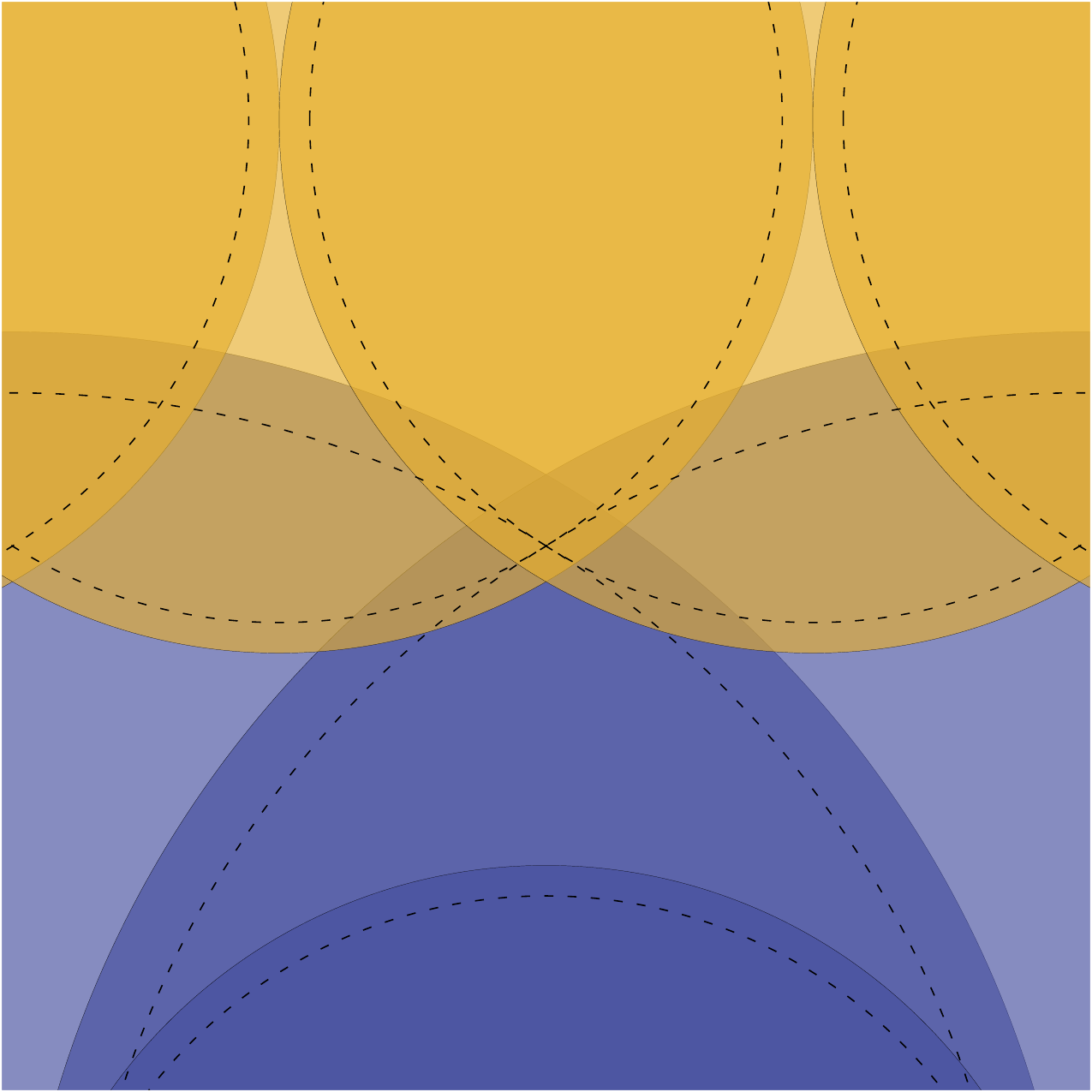}
\caption{Left:  discs of radius $1$ and $1/2$ on half-integers with $|\Delta|=23$; now $z_3$ is covered. Right: scaling of radii by $\sqrt{8/9}$.}\label{fig:3}
\end{figure}

Unfortunately, continuing in this fashion does not really work. Convergents converge to $z$, but they may not come close in quality to the approximations that must exist by Dirichlet's box principal. The missing piece is a bound on $|\det M_n|$, which can grow exponentially when $|b|\neq 1$ in $S(a,b)$. So we make an adjustment: since $\det M_4=b_4=2$, in the next stage we pick among matrices of the form $S(a/2,b/2)$, where $a\in \cO$ and $b\in\{1,2\}$. This cancels the previous determinant, and $|\det M_5|=b_5\in\{1,2\}$ again. Since the goal is to approximate $z$ with ratios of integers, $a_5$ and $b_5$ are now subject to the restriction that $M_4S(a_5/2,b_5/2)$ be integral. Matrix multiplication shows this condition is equivalent to (\ref{eq:101}). Our new divisibility requirement eliminates half of the discs in Figure \ref{fig:3}. But the ones that survive now get a disc of radius $2/b$ instead of $1/b$, as in (\ref{eq:100}). We need this to remain a covering to guarantee containment of $z_4$. It does, as can be seen in the first image of Figure \ref{fig:4}. That we continue to obtain a covering using $B=\{1,2\}$ in subsequent stages of the algorithm makes this set admissible for $\bQ(\sqrt{-23})$.

It is not uncommon that a set $B$ produces a covering at one stage (like Figure \ref{fig:3} for the fourth stage in the example) but not another. A few examples of such inadmissible sets are $\{1,(1+\sqrt{-15})/2\}$ for $\Delta=-15$, $\{1,2\}$ for $\Delta=-31$, $-39$, or $-47$, and $\{1,2,(1\pm\sqrt{-35})/2\}$ for $\Delta=-35$.

There is one subtlety regarding coefficient choice that occurs if $B\neq \{1\}$. In our example from $\bQ(\sqrt{-23})$, note that if $a_n$ and $b_n=1$ make $M_{n-1}S(a_n/b_{n-1},1/b_{n-1})$ integral and $|z_{n-1}-a_n|<1/2$, then we might instead choose $2a_n$ and $b_n=2$. Indeed, $M_{n-1}S(2a_n/b_{n-1},2/b_{n-1})$ is integral and $|2z_{n-1}-2a_n|<1$. This doubles the resulting values of $p_n$ and $q_n$, presenting a potential problem: the undoubled values may appear at a later index, meaning the same convergent could occur twice. This would necessitate unpleasant caveats in several of Section \ref{sec:3}'s results. As such, we insist that $(p_n,q_n)$ be reduced to the extent that avoids this issue.

\begin{definition}\label{def:reduce}For $\eps\in(0,1)$, an ideal $\fb\subseteq\cO$ is $\eps$-\emph{reduced} if for every $k\in K\backslash\{0\}$, $k\fb\subseteq\cO$ implies $|k| > \eps^2$.\end{definition}

The relation between $\eps$ in Definition \ref{def:reduce} and $B$ is clarified shortly.

\subsection{The algorithm}\label{ss:alg}Definition \ref{def:admit} formalizes the covering requirement discussed in the previous subsection. For computations, this definition can be skipped in favor of Table \ref{table:2} or Theorem \ref{thm:ints}.

\begin{notation}\label{not:disc}Let $D(z,r)$ denote the closed disc of radius $r>0$ and center $z\in\bC$.\end{notation}

\begin{definition}\label{def:admit}A nonempty, finite set $B\subset\cO\backslash\{0\}$ is \emph{admissible} with $\eps\in(0,1)$ if for every $\eps$-reduced ideal $\fb$ with $\fb\cap B\neq \emptyset$, $$\bC=\bigcup_{a,b}D\!\left(\frac{a}{b},\frac{\eps}{|b|}\right),$$ where the union ranges over the pairs $a,b$ for $a\in K$ and $b\in B$ that make $(a\fb,b\fb^{-1})$ integral and $\eps$-reduced. (Here $\fb^{-1}$ is the fractional ideal satisfying $\fb^{-1}\fb=\cO$.)\end{definition}

The value of $\eps$ in Definition \ref{def:admit} is a guaranteed measure of approximation quality improvement, $|q_nz-p_n|\leq\eps|q_{n-1}z-p_{n-1}|$. Geometrically, it is an allowable amount by which radii of discs can be scaled while preserving the covering, as shown in Figure \ref{fig:3}. 

Note that Definition \ref{def:admit} does not mention $b_{n-1}$, $p_{n-1}$, $q_{n-1}$, $p_{n-2}$, or $q_{n-2}$, all of which appear in (\ref{eq:101}) and therefore determine the coverings used by Algorithm \ref{alg:1}. Since there are infinitely many values these variables might attain, a practical definition of admissibility should adjust for the redundancy of checking every potential covering. Definition \ref{def:admit} requires that $\fb$ be $\eps$-reduced, so the number of coverings checked is bounded by a small multiple of the class number (or exactly the class number if $\eps$ is sufficiently close to $1$). This facilitates proofs of admissibility and searches for admissible sets. Unfortunately, it also obscures the relationship between admissibility and Algorithm \ref{alg:1}. For example, it is likely not clear at this point why coverings indexed by $\fb$ suffice. And while ``$a$" from Definition \ref{def:admit} is directly related to its counterpart in lines 4--6 of Algorithm \ref{alg:1}, they are not equal. The precise connection is postponed until Section \ref{sec:4}.

It may be useful to first consider Algorithm \ref{alg:1} in a Euclidean ring with $B=\{1\}$. The \textbf{if} condition in line 5 becomes trivial and can be ignored. It is then the Hurwitz algorithm with the exception that we are not requiring $a_n$ to be the nearest integer to $z_{n-1}$, only that $|z_{n-1}-a_n|<\eps$.

\begin{definition}\label{def:ideal}The \emph{left-column ideal} of a matrix $M$, denoted $(M)_{\ell}$, is the ideal generated by its left-column entries. Define the \emph{right-column ideal}, $(M)_r$, similarly.\end{definition}

\begin{algorithm}[H]\caption{Compute continued fraction convergents of $z\in\bC$ over $\cO$. Any method (like Algorithm \ref{alg:2}) for choosing among multiple coefficient pairs $a,b$ satisfying line 5 may be used.}\label{alg:1}
\begin{flushleft}
\hspace*{\algorithmicindent}\textbf{input:} $z\in\bC$, $N\in\bN$, $B$ admissible with $\eps\in(0,1)$ as per Definition \ref{def:admit}\\
\hspace*{\algorithmicindent}\textbf{output:} $p_N,q_N\in\cO$ with $p_N/q_N$ approximating $z$ 
\end{flushleft}
\begin{algorithmic}[1]
    \State $M\gets\text{Id}\in\text{SL}_2(\cO)$\Longtop{convergents are column ratios}
    \State $b'\gets 1$\Bottom{of $M$ as in (\ref{eq:2})}
    \For{$n\gets 1$ \textbf{to} $N$}
    \For{$b\!\in\!B$, $a\!\in\!\cO\cap\! D(bM^{-1}(z),\eps|b'|)$}\Longtop{$M^{-1}$ is M\"{o}bius transformation}
    \If{$(MS(a/b',b/b'))_{\ell}$ is integral and \Longtop{Algorithm \ref{alg:2} gives a subroutine}\Statex \hspace{\algorithmicindent}\hspace{\algorithmicindent}$\eps$-reduced}\Bottom{that finds $a,b$ for certain $B$}
        \State $M\gets MS(a/b',b/b')$
        \State $b'\gets b$
        \State \textbf{break}
        \Statex \vspace{-\baselineskip}
    \EndIf
    \EndFor
    \State\textbf{if} $z = M_{1,1}/M_{2,1}$ \textbf{then break}\Top{we found $z$ exactly}
    \Statex \vspace{-\baselineskip}
    \EndFor
    \State \Return $M_{1,1},M_{2,1}$
\end{algorithmic}
\end{algorithm}

\begin{notation}\label{not:vars}Let $a_n$, $b_n$, and $M_n$ denote ``$a$," ``$b'$," and ``$M$" after completing the $n^{\text{th}}$ outer \textbf{for} loop iteration, with $b_0$ and $M_0$ being initial values, let $z_n=M_n^{-1}(z)$, and let $p_n$ and $q_n$ denote the left column entries of $M_n$. Its right column entries are then $p_{n-1}$ and $q_{n-1}$, which we use to define $p_{-1}=0$ and $q_{-1}=1$.\end{notation}

It follows from line 6 that our variables satisfy the same relations that hold in Euclidean cases when $B=\{1\}$. (Results do not mention the input $N$ or whatever the terminating index happens to be.)

\begin{proposition}\label{prop:rels}If $n\geq 1$ then $$p_n=\frac{a_np_{n-1}+b_np_{n-2}}{b_{n-1}},\hspace{1cm}q_n=\frac{a_nq_{n-1}+b_nq_{n-2}}{b_{n-1}},\hspace{1cm}z_n=\frac{b_{n-1}}{b_nz_{n-1}-a_n},$$  $$\frac{p_n}{q_n} = \frac{a_1}{b_1} +\cfrac{b_0/b_1}{\cfrac{a_2}{b_2}+\cfrac{b_1/b_2}{\raisebox{6pt}{$\ddots$} \; \cfrac{a_{n-1}}{b_{n-1}}+ \cfrac{b_{n-2}/b_{n-1}}{a_n/b_n}}}\,,$$ and $\det M_n=(-1)^nb_n$.\end{proposition}

\begin{proof}The expressions for $p_n$, $q_n$, $z_n$, and $\det M_n$ follow directly from line 6 (and induction for $\det M_n$). Viewing our matrices as M\"{o}bius transformations, from the new expressions for $p_n$ and $q_n$ we see that $$\frac{p_n}{q_n} = M_{n-1}\!\left(\frac{a_n}{b_n}\right)=\bigg(S\!\left(\frac{a_1}{b_0},\frac{b_1}{b_0}\right)\circ\cdots \circ S\!\left(\frac{a_{n-1}}{b_{n-2}},\frac{b_{n-1}}{b_{n-2}}\right)\bigg)\left(\frac{a_n}{b_n}\right).$$ The continued fraction given in the proposition is an expansion of the right-hand side since $S(a,b)(z) = a/b + 1/bz$.\end{proof}

\subsection{An example}\label{ss:example} Recall the example in Subsection \ref{ss:intuit} for $\bQ(\sqrt{-23})$. It starts with $z=-1.26+0.48i$ and $B=\{1,2\}$. Let $\eps=\sqrt{8/9}$ and $\tau=(1+\sqrt{-23})/2$.

Prior choices of coefficients are $a_1=-2$, $a_2=1$, and $a_3=-1+\tau$, which center the outlined discs in Figure \ref{fig:1} that contain $z_0$, $z_1$, and $z_2$. We claim these still meet the requirements of Algorithm \ref{alg:1} alongside $b_1=1$, $b_2=1$, and $b_3=1$. Indeed, when $b_{n-1}=b_n=1$, the disc containment in line 4 is the same as $z_{n-1}\in D(a_n,\eps)$. In our example, the radii in Figure \ref{fig:1} can be scaled by $\eps$ and still cover $z_0$, $z_1$, and $z_2$. Moreover, line 5's requirement that $(M_{n-1}S(a_n/1,1/1))_{\ell}=(M_n)_{\ell}$ be $\eps$-reduced is satisfied since $\det M_n=\pm1$ implies $(M_n)_{\ell}=\cO$.

So we keep our original three coefficients. Starting with the identity matrix, $M_0$, line 6 gives $$M_1=\begin{bmatrix}-2 & 1 \\ 1 & 0\end{bmatrix},\hspace{1cm}M_2=\begin{bmatrix}-1 & -2 \\1  & 1\end{bmatrix},\hspace{0.5cm}\text{and}\hspace{0.5cm}M_3=\begin{bmatrix}-1-\tau & -1 \\ \tau & 1\end{bmatrix}.$$

The previously discussed choice of $a_4=\tau$ and $b_4=2$ also passes the \textbf{if} condition in line 5. Indeed, it gives $$M_4=\begin{bmatrix}4-2\tau & -1-\tau \\ -4+\tau & \tau\end{bmatrix},$$ and thus $(M_4)_{\ell}=(\tau,2)$. This is a (split) prime over $2$, which is $\eps$-reduced regardless of the value of $\eps$. We get $z_4=M_4^{-1}(z)\approx 1.43+0.96i$.

Consider the top row of $M_4$. We cannot use $b_5=1$ because $a(4-2\tau) + 1(-1-\tau)$ is not divisible by $b_4=2$ for any $a\in\bZ[\tau]$. As $b_5$ must come from $\{1,2\}$, $b_5=2$ is forced. So line 4 looks for $a_5\in\cO \cap D(2z_4,2\eps)$, which rearranges to $z_4\in D(a_5/2,\eps)$. Turning to the second row of $M_4$, $a_5(-4+\tau) + 2\tau$ is divisible by $2$ if and only if $a_5\in(\overline{\tau},2)$. The first image of Figure \ref{fig:4} shows that unit discs on $a/2$ for $a\in (\overline{\tau},2)$ do indeed cover the plane with radius-scaling room to spare. In particular, we may take $a_5=1+\tau$. The highlighted disc is $D((1+\tau)/2,1)$.

\begin{figure}
\centering
\begin{overpic}[unit=1mm,trim=2cm 0cm 2cm 0cm,clip,height=5.8cm]{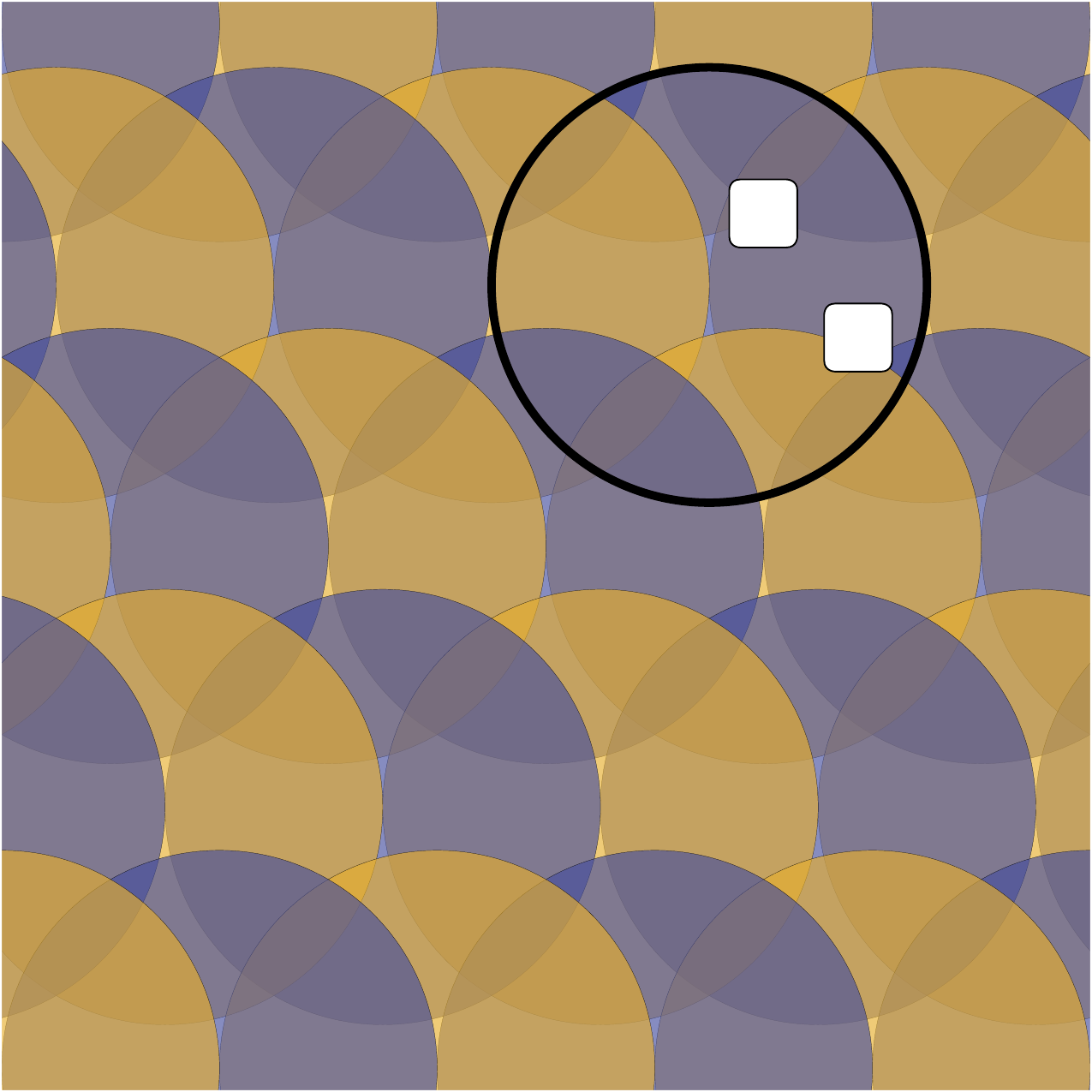}
    \put(39.3,39.01){\small$4$}
    \put(34.36,45.4){\small$\overline{6}$}
\end{overpic}\hspace{0.5cm}
\begin{overpic}[unit=1mm,trim=2cm 0cm 2cm 0cm,clip,height=5.8cm]{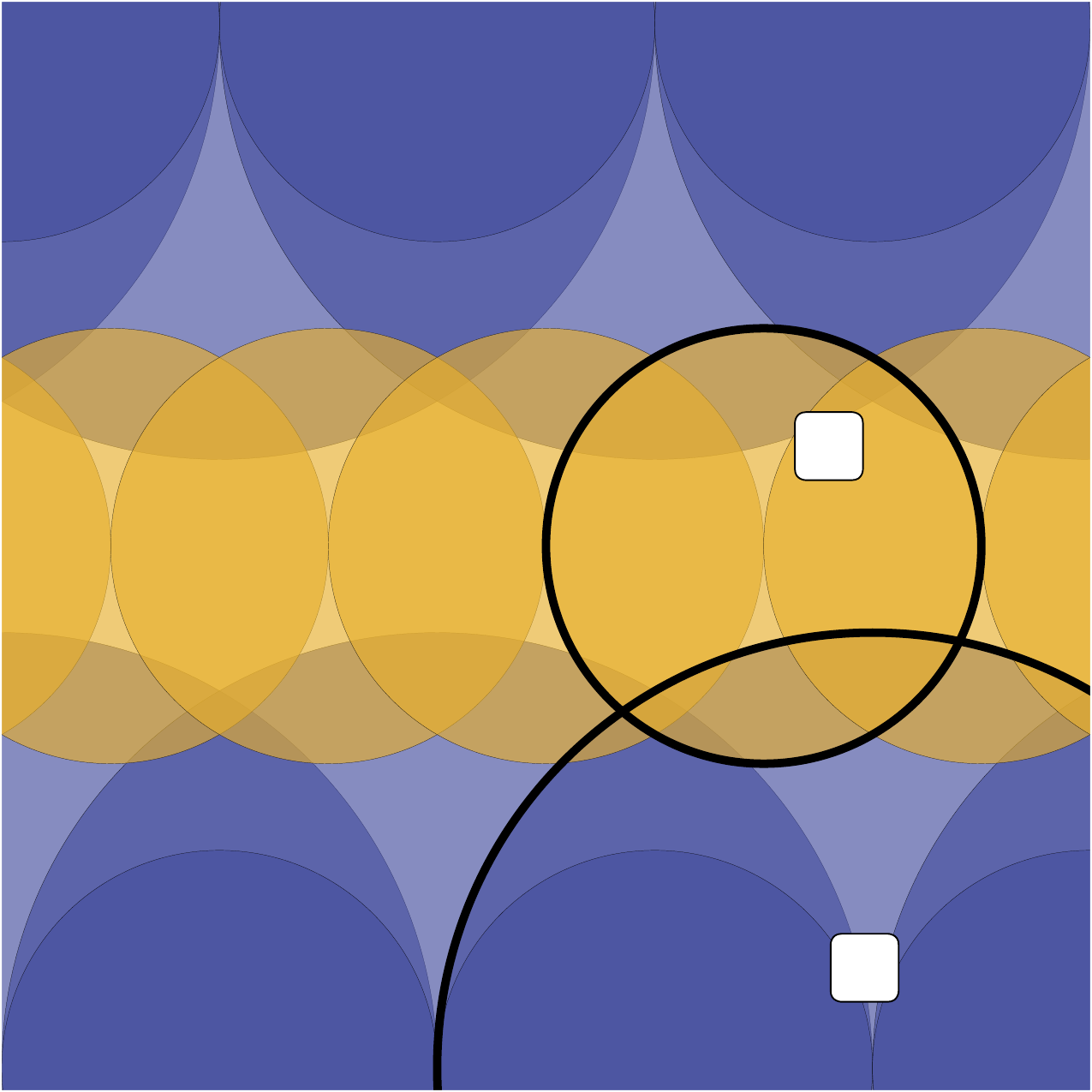}
    \put(37.8,33.25){\small$5$}
    \put(39.7,5.35){\small$\overline{7}$}
\end{overpic}
\caption{Left: $z_4$ and $\overline{z_6}$ in a disc on $a_5/b_5=\overline{a_7/b_7}=(1+\tau)/2$. Right: $z_5$ and $\overline{z_7}$ in discs on $a_6/b_6=1$ and $\overline{a_8/b_8}=2-\tau$.}\label{fig:4}
\end{figure}

The congruence requirement on $a$ and $b$ can be computed similarly from $$M_5=\begin{bmatrix}7-\tau & 4-2\tau\\-5 & -4+\tau\end{bmatrix}.$$ It is $a\equiv\tau\,\text{mod}\,2$ if $b=1$, and $a$ can be any integer if $b=2$. (But $a\equiv 2\tau\,\text{mod}\,4$ and $b=2$ needs to be reduced to $a/2$ and $b/2$ according to Definition \ref{def:reduce} because $1/2\leq \eps^2 = 8/9$.) The corresponding discs of radius $b_5/b=2/b$ and center $a/b$ are displayed in the second image of Figure \ref{fig:4}. We see that $a_6=b_6=2$ satisfies $z_5\in D(a_6/b_6,|b_5/b_6|) = D(1,1)$, again with room to scale radii by $\eps$.

The arrangement of discs that occurs for $z_6$, $\cup D(a/b,2/b)$ for $a\in\cO$ and $b\in\{1,2\}$ that make $(M_6S(a/b_5,b/b_5))_{\ell}$ $\eps$-reduced, is the vertical reflection of $z_4$'s. So the first image of Figure \ref{fig:4} shows $\overline{z_6}$ in the disc centered on a possible choice of $\overline{a_7/b_7}$, which happens to be the same disc we chose for $z_4$. We are also able to squeeze $\overline{z_7}$ into $z_5$'s image.

The resulting convergents for $n\leq 10$ are given in Table \ref{table:1} along with approximation quality. It can be checked that $|(q_nz-p_n)| < \eps|q_{n-1}z-p_{n-1}|$ with $\eps=\sqrt{8/9}$, a direct result of $|b_nz_{n-1}-a_n|<\eps|b_{n-1}|$.

\setlength\tabcolsep{0.22cm}
\renewcommand{\arraystretch}{1.2}

\begin{table}[ht]
\setlength{\abovecaptionskip}{-0.1\baselineskip}
\setlength{\belowcaptionskip}{0.4\baselineskip}
\begin{tabular}{|c|c|c|c|c|}
\hline\rowcolor{Gray}
$n$ & $\approx z_{n-1}$ & $a_n/b_n$ & $p_n/q_n$ & $\approx|q_nz-p_n|$\\ \hline
1 & $-1.26+0.48i$ & $-2/1$ & $-2/1$ & $0.882$\\ \hline
2 & $0.95-0.62i$ & $1/1$ & $-1/1$ & $0.5459$\\ \hline
3 & $-0.13+1.61i$ & $(-1+\tau)/1$ & $(-1-\tau)/\tau$ & $0.4754$\\ \hline
4 & $0.49+1.04i$ & $\tau/2$ & $(4-2\tau)/(-4+\tau)$ & $0.2757$\\ \hline
5 & $1.43+0.96i$ & $(1+\tau)/2$ & $(7-\tau)/(-5)$ & $0.2$\\ \hline
6 & $1.3+0.46i$ & $2/2$ & $(11-3\tau)/(-9+\tau)$ & $0.1096$\\ \hline
7 & $1-1.53i$ & $(2-\tau)/2$ & $(9-8\tau)/(-11+5\tau)$ & $0.0451$\\ \hline
8 & $1.46+1.94i$ & $(1+\tau)/1$ & $(34-5\tau)/(-25)$ & $0.0104$\\ \hline
9 & $-0.34+4.32i$ & $(-2+2\tau)/1$ & $(1+60\tau)/(39-45\tau)$ & $0.0085$\\ \hline
10 & $0.99+0.72i$ & $1/1$ & $(35+55\tau)/(14-45\tau)$ & $0.0061$\\ \hline
\end{tabular}
\captionsetup{width=.93\linewidth}
\caption{Coefficients, convergents, and approximation quality from Algorithm \ref{alg:1} with $\Delta=-23$ using $B=\{1,2\}$ and $\eps = \sqrt{8/9}$.}\label{table:1}
\end{table}

\renewcommand{\arraystretch}{1}

Observe that the last two continuants satisfy $|q_9|^2=11916$ and $|q_{10}|^2=11716$. For classical continued fractions and Hurwitz' algorithm over the Euclidean rings, continuant magnitudes increase monotonically. This fails in general. But Theorem \ref{thm:growq} asserts that the degree to which continuant monotonicity fails is bounded by a constant depending only on $B$ and $\eps$.

We end this section with a remark on the colors in Figures \ref{fig:3} and \ref{fig:4}. The second image of Figure \ref{fig:4} is a scaled and shifted copy of Figure \ref{fig:3}. It turns out that up to scaling, shifting, and reflecting, the two disc arrangements in Figure \ref{fig:4} are the only ones that can occur. (Such similarity of arrangements is how we get away with the apparently scant number of coverings provided by Definition \ref{def:admit}.) Colors foreshadow which of the two types of arrangement occurs next: yellow for the first image in Figure \ref{fig:4} and blue for the second. Our choice of discs containing $z_0$, $z_1$, $z_2$, $z_4$, $z_6$, and $z_7$ are blue, so it is a scaled or shifted copy of the second image in Figure \ref{fig:4} that must cover $z_1$, $z_2$, $z_3$, $z_5$, $z_7$, and $z_8$. Since yellow discs are chosen to cover $z_3$ and $z_5$, a scaled, shifted, or reflected copy of the first image in Figure \ref{fig:4} must cover $z_4$ and $z_6$. 

Which of the two disc arrangements (either the first or second image in Figure \ref{fig:4}) appears in stage $n$ is determined by the ideal class of $(M_{n-1})_{\ell}$---trivial is the second image and nontrivial the first. So when drawing discs, the appropriate color for $D(a/b,b/b')$ can determined by computing what the ideal class of $(M_n)_{\ell}$ would be if $a$ and $b$ were selected as coefficients. That is, we compute the ideal class of $(M_{n-1}S(a/b',b/b'))_{\ell}$---trivial gets blue and nontrivial gets yellow.  Note that there are two nontrivial ideal classes for $\bQ(\sqrt{-23})$. But they are inverses, implying complex conjugation maps an ideal in one class to an ideal in another. The two disc arrangements that occur when $(M_{n-1})_{\ell}$ is conjugate are vertical reflections of one another (perhaps scaled or shifted as well). This is why the the first image in Figure \ref{fig:4} may appear reflected at future stages, as it is for covering $z_6$ in stage 7. The second image in Figure \ref{fig:4} is preserved by conjugation, as is true of the class of principal ideals.

It would be interesting to study whether the sequence of ideal classes of $(M_n)_{\ell}=(p_n,q_n)$, rather than the actual convergents $p_n/q_n$, still carries information about the input $z$.

\section{Classical Properties}\label{sec:3} This section rifles through Hensley's litmus test for continued fractions (Section 5.2 of \cite{hensley}). Essentially, properties of the nearest integer algorithm over $\bZ$ are retained at the expense of constants (meaning with respect to $n$ and $z$) that grow with the discriminant magnitude $|\Delta|$.

\subsection{Convergents}\label{ss:conv} The following observation is often used without mention.

\begin{lemma}\label{lem:z}If $n\geq 1$ then $|z_n| \geq 1/\eps$.\end{lemma}

\begin{proof}By Proposition \ref{prop:rels}, $z_n=b_{n-1}/(b_nz_{n-1}-a_n)$. So $|z_n|\geq 1/\eps$ is equivalent to $a_n\in D(b_nz_{n-1},\eps|b_{n-1}|)$, which is from line 4 of Algorithm \ref{alg:1}.\end{proof}

\begin{proposition}\label{prop:mono}If $n\geq 1$ then $|q_nz - p_n| \leq \eps|q_{n-1}z - p_{n-1}|$. In particular, $|q_nz-p_n|\leq\eps^n$.\end{proposition}

\begin{proof}Notation \ref{not:vars} defines $z_n$ to be $M_n^{-1}(z)=(q_{n-1}z-p_{n-1})/(p_n-q_nz)$. So the first inequality in the proposition is equivalent to $|z_n|\geq 1/\eps$, which is the previous lemma. The second assertion follows by induction.\end{proof}

\begin{corollary}\label{cor:nzero}If $n\geq 1$ then $a_{n+1}$ and $q_n$ are nonzero, and $n$ marks the first occurrence of $p_n/q_n$ as a convergent.\end{corollary}

\begin{proof}That $q_n\neq 0$ follows from $p_n\in\cO$ and the second assertion of Proposition \ref{prop:mono}: $|q_nz-p_n|\leq\eps^n<1$. 

If $p_{n'}/q_{n'}=p_n/q_n$ with $n\geq n'$ then $(q_n/q_{n'})(p_{n'},q_{n'}) = (p_n,q_n)\subseteq\cO$. Proposition \ref{prop:mono} gives $\eps^{n-n'}\geq|q_nz-p_n|/|q_{n'}z-p_{n'}|=|q_n/q_{n'}|$, which implies $n-n' < 2$ by Definition \ref{def:reduce} because $(p_{n'},q_{n'})$ is $\eps$-reduced. For the case $n = n'+1$ we have $p_nq_{n-1}-p_{n-1}q_n = b_n\neq 0$ by Proposition \ref{prop:rels}. 

Now that we know $p_{n+1}/q_{n+1}\neq p_{n-1}/q_{n-1}$, Proposition \ref{prop:rels}'s formulas for $p_{n+1}$ and $q_{n+1}$ show that $a_{n+1}\neq 0$ for $n\geq 1$.\end{proof}

\begin{corollary}\label{cor:term}If $z=p/q$ for $p,q\in\cO$, then $p_n/q_n = z$ for some $n\leq \lfloor 1-\log_{\eps}\!|q|\rfloor$.\end{corollary}

\begin{proof}By Proposition \ref{prop:mono}, the integer $q(q_nz-p_n)=q_np-p_nq$ is bounded in magnitude by $\eps^n|q|$. Setting this equal to $1$ and solving shows that $q_np-p_nq=0$ no later than $n= \lfloor 1-\log_{\eps}\!|q|\rfloor$.\end{proof}

\begin{notation}For a fixed admissible set $B$, let $\mu = \max_B\!|b|$.\end{notation}

The following lemma and theorem both have a counterparts (which we are not yet ready to prove), Lemma \ref{lem:up} and Theorem \ref{thm:mainb}, where the directions of the inequalities are reversed.

\begin{customlem}{3.6a}\label{lem:low}\addtocounter{theorem}{1}If $n\geq 1$ then $$\left|1 + \frac{q_{n-1}}{q_nz_n}\right| > \frac{(1-\eps^2)|b_n|}{\mu}.$$\end{customlem}

\begin{proof}The proof proceeds by induction. The base case is $n=1$, which holds using $q_0=0$, $|b_1|\leq\mu$, and $\eps > 0$.

Let $n>1$ and assume the lemma's inequality holds when each index is decreased by one. Since $|z_n|\geq 1/\eps$ and $|b_n| \leq\mu$, the claim holds if $|q_{n-1}/q_n|< \eps$ by the triangle inequality---no need for induction. Otherwise, by Proposition \ref{prop:rels} we have $$\frac{1+q_{n-1}/q_nz_n}{b_n}=\frac{q_n+q_{n-1}/z_n}{b_nq_n} = \frac{(a_nq_{n-1}+b_nq_{n-2})+q_{n-1}(b_nz_{n-1}-a_n)}{b_{n-1}b_nq_n}=$$ \begin{equation}\frac{q_{n-2}+q_{n-1}z_{n-1}}{b_{n-1}q_n}=\frac{q_{n-1}z_{n-1}}{q_n}\left(\frac{1+q_{n-2}/q_{n-1}z_{n-1}}{b_{n-1}}\right).\label{eq:6}\end{equation} But we assumed $1\leq |q_{n-1}/\eps q_n|$, which is at most $|q_{n-1}z_{n-1}/q_n|$ by Lemma \ref{lem:z}. So the final expression in (\ref{eq:6}) is at least $|1+q_{n-2}/q_{n-1}z_{n-1}|/|b_{n-1}|$ in magnitude, which exceeds $(1-\eps^2)/\mu$ by the induction hypothesis.\end{proof}

\begin{customthm}{3.7a}\label{thm:main}\addtocounter{theorem}{1}If $n\geq 1$ then $|q_nz-p_n|$ is less than
$$i)\;\; \frac{\mu}{(1-\eps^2)|q_nz_n|},\hspace{1cm}ii)\;\;\frac{\mu}{(1-\eps^2)|q_{n+1}|},\hspace{0.5cm}\text{and}\hspace{0.5cm}iii)\;\;\frac{(1+\eps^2)\mu^2}{(1-\eps^2)|a_{n+1}q_n|}.$$\end{customthm}

\begin{proof}Consider the identity \begin{equation}q_nz - p_n = q_nM_n(z_n) - p_n =\frac{-\det M_n}{q_nz_n(1+q_{n-1}/q_nz_n)}.\label{eq:7}\end{equation} Since $\det M_n = (-1)^nb_n$, we see that $i)$ is a rearrangement of Lemma \ref{lem:low}. 

To prove $ii)$, increment the index in (\ref{eq:6}) and scale both ends by $q_{n+1}$ to get $$\frac{q_{n+1}(1+q_n/q_{n+1}z_{n+1})}{b_{n+1}}=\frac{q_nz_n(1+q_{n-1}/q_nz_n)}{b_n}.$$ Up to a sign, the reciprocal of the right-hand side above equals the final expression in (\ref{eq:7}). So we extend (\ref{eq:7}) as follow: \begin{equation}\frac{-\det M_n}{q_nz_n(1+q_{n-1}/q_nz_n)}=\frac{(-1)^{n+1}b_{n+1}}{q_{n+1}(1+q_n/q_{n+1}z_{n+1})}.\label{eq:8}\end{equation} Lemma \ref{lem:low} bounds the magnitude of the last expression by $\mu/(1-\eps^2)|q_{n+1}|$.

Finally, $|b_{n+1}z_n-a_{n+1}|\leq\eps|b_n|$ gives $|a_{n+1}| \leq |b_{n+1}z_n| + \eps|b_n|\leq (|z_n|+\eps)\mu\leq(1+\eps^2)\mu|z_n|$. This provides a lower bound for $|z_n|$ in terms of $|a_{n+1}|$ that can be substituted into $i)$ to prove $iii)$.\end{proof}

As with classical continued fractions, if $|q(qz-p)|$ is sufficiently small for $p,q\in\cO$, then $p/q$ appears as a convergent in the expansion of $z$. This property is interesting in our case because we are not specifying any particular method for choosing among multiple pairs $a,b$ passing the \textbf{if} condition in line 5. As such, the following lemma asserts that finding sufficiently good approximations is unavoidable.

\begin{lemma}\label{lem:combo}Let $p,q\in\cO$ with $q\neq 0$ and $z\neq p/q$. Take $n\geq -1$ to be the smallest index for which \begin{equation}\label{eq:9}|q_{n+1}|\geq \sqrt{\frac{|q|\mu}{|qz-p|(1-\eps^2)}},\end{equation} or, if no such index exists, take $n$ so that $p_n/q_n=z$. Then there exist $a,b\in\cO$ with $$|b|< 2\sqrt{\frac{|q(qz-p)|\mu}{(1-\eps^2)}}$$ satisfying $p=(ap_n+bp_{n-1})/b_n$ and $q=(aq_n+bq_{n-1})/b_n$. In particular, $p/q=p_n/q_n$ whenever $|q(qz-p)|\leq(1-\eps^2)/4\mu$.\end{lemma}

\begin{proof}For a given $n$, the value of $b\in\cO$ that makes $p=(ap_n+bp_{n-1})/b_n$ and $q=(aq_n+bq_{n-1})/b_n$ for the right choice of $a\in\cO$ is $b=p_nq-pq_n$. 

First suppose there exists an index for which (\ref{eq:9}) holds, and let $n$ be the smallest. The first inequality below is the triangle inequality, the second is Theorem \ref{thm:main} $ii)$, and the third comes from $|q_{n+1}|$ satisfying (\ref{eq:9}) but $|q_n|$ violating (\ref{eq:9}) by minimality of $n$:  $$|b|=|p_nq-pq_n|\leq |q_n(qz-p)|+|q(q_nz-p_n)|<$$ $$|q_n(qz-p)|+ \left|\frac{q}{q_{n+1}}\right|\frac{\mu}{(1-\eps^2)}<2\sqrt{\frac{|q(qz-p)|\mu}{(1-\eps^2)}}.$$

If (\ref{eq:9}) is never satisfied then $z$ must be rational because otherwise Proposition \ref{prop:mono} implies continuants grow without bound. By Corollary \ref{cor:term} we we can choose $n$ with $p_n/q_n=z$. The assumption that $|q_n|$ does not satisfy (\ref{eq:9}) bounds it from above. We use this upper bound for the inequality below: $$|b|=|p_nq-pq_n|=\left|\frac{q_n}{q}\right||q(qz-p)|<\sqrt{\frac{|q(qz-p)|\mu}{(1-\eps^2)}}.$$ 

For the last claim, combining $|q(qz-p)|\leq(1-\eps^2)/4\mu$ with the upper bound on $b$ in the statement of the lemma forces $|b|<1$. Thus $b=0$, implying $p/q=(ap_n/b_n)/(aq_n/b_n)=p_n/q_n$ as claimed.\end{proof}

\begin{theorem}\label{thm:best}If $p/q$ is not a convergent of $z$ for some $p,q\in\cO$ with $q\neq 0$, then $$|q_n(q_nz-p_n)|<\frac{4\eps\mu^2|q(qz-p)|}{(1-\eps^2)^2}$$ for any $n\geq 1$. That is, each $p_n/q_n$ is a best approximation of the second kind up to constants: if $rs\leq (1-\eps^2)^2/4\eps\mu^2$, then $0<|q|<r|q_n|$ implies $|qz-p|>s|q_nz-p_n|$ for any $p\in\cO$ except perhaps when $p/q$ is already a convergent.\end{theorem}

\begin{proof}Violating the stated inequality combines with Theorem \ref{thm:main} $i)$, giving $$|q(qz-p)|\leq \frac{(1-\eps^2)^2|q_n(q_nz-p_n)|}{4\eps\mu^2}< \frac{1-\eps^2}{4\mu}.$$ Thus $p/q$ is a convergent by Lemma \ref{lem:combo}.\end{proof}

\subsection{Continuants}

Here we study the growth of $|q_n|$, which may not be monotonic as the example in Subsection \ref{ss:example} demonstrates.

\begin{lemma}\label{lem:er}Suppose $B$ and $\eps$ are admissible and that $\mu\neq 1$ or $|\Delta|\neq 3$. Then $\eps\mu\geq 2/3$.\end{lemma}

\begin{proof}Let $z\in\bR$ approach $\eps$ from the right. Every nonzero element of $\cO$ has magnitude at least $1$, so no nonzero multiple of $z$ is within $\eps$ of $0$. In particular, the minimal nonzero multiple of $z$ that is within $\eps$ of any integer is $\lceil 1/\eps - 1\rceil z$. So $\mu\geq \lceil 1/\eps - 1\rceil$, which is at least $2/3\eps$ unless $\eps\in[1/2,2/3)$. In this range, $\eps\mu< 2/3$ would imply $\mu < 4/3$. This forces $B$ to consist of units since non-units in imaginary quadratic rings have magnitude at least $\sqrt{2}$. Thus the discs in Definition \ref{def:admit}'s union are centered on integers. But discs of radius $\eps\in[1/2,2/3)$ on integers only cover the plane when $\Delta=-3$.\end{proof}

\begin{theorem}\label{thm:growq}If $\,0\leq n'<n$, then $$|q_n|> \frac{(1-\eps^2)^2|q_{n'}z_{n'}|}{4\eps^{n-n'}\mu^2}.$$ In particular, if $n\geq 1$ then $|q_n|>(1-\eps^2)^2/4\eps^n\mu^2$.\end{theorem}

\begin{proof}Since $q_0=0$, we may assume that $n'\geq 1$. Suppose the first lower bound on $|q_n|$ is false. Then by Proposition \ref{prop:mono} and Theorem \ref{thm:main} $i)$, $$|q_n(q_nz-p_n)|<\eps^{n-n'}|q_n(q_{n'}z-p_{n'})|\leq \frac{(1-\eps^2)^2|q_{n'}z_{n'}(q_{n'}z-p_{n'})|}{4\mu^2}<\frac{1-\eps^2}{4\mu}.$$ Therefore Lemma \ref{lem:combo} applies, and either $p_n/q_n=z$ or $n=n''$, where $n''$ is the first index (recall convergents cannot repeat by Corollary \ref{cor:nzero}) for which $$|q_{n''+1}|\geq \sqrt{\frac{|q_n|\mu}{|q_nz-p_n|(1-\eps^2)}}.$$ Regarding the second possibility, $|q_{n-1}|$ must fail to satisfy the bound above in place of $|q_{n''+1}|$. Thus $$1\leq |p_nq_{n-1}-p_{n-1}q_n|\leq |q_n(p_{n-1}-q_{n-1}z)|+|q_{n-1}(q_nz-p_n)|$$ $$<\frac{\eps^{n-n'}|q_n(q_{n'}z-p_{n'})|}{\eps} + \sqrt{\frac{|q_n(q_nz-p_n)|\mu}{1-\eps^2}}< \frac{1-\eps^2}{4\eps\mu}+ \frac{1}{2}<1.$$ The last inequality uses Lemma \ref{lem:er}. The same contradiction occurs in the case $p_n/q_n=z$; we just get to replace the summand $|q_{n-1}(q_nz-p_n)|$ above with $0$.

Finally, $|q_n|>(1-\eps^2)^2/4\eps^n\mu^2$ uses $n'=1$ and $|q_{n'}z_{n'}|\geq 1\cdot 1/\eps$.\end{proof}

\begin{corollary}If $n\geq 1$, then $$\left|z-\frac{p_n}{q_n}\right|<\frac{4\eps^{2n}\mu^2}{(1-\eps^2)^2}.$$\end{corollary}

\begin{proof}This follows from $|q_n|>(1-\eps^2)^2/4\eps^n\mu^2$ and $|q_nz-p_n|\leq\eps^n$.\end{proof}

\subsection{Coefficients}\label{ss:coef}Here we show that the existence of an infinite, periodic sequence of coefficients for an input $z$ is equivalent to $[K(z):K]=2$, and that boundedness of coefficients is equivalent to $z$ being badly approximability. Both are true whether ``coefficient" is interpreted to mean $a_n$, $a_n/b_n$, or the pair $a_n,b_n$, and there is little difference made to the proofs.

\begin{definition}We call $z\in\bC$ \emph{badly approximable} if $|q(qz-p)|$ has a positive infimum over $p,q\in\cO$ with $q\neq 0$. Otherwise, $z$ is \emph{well approximable}.\end{definition}

To prove that bad approximability is equivalent to bounded coefficients, we need a lower-bound analogue of Theorem \ref{thm:main} $iii)$.

\begin{customlem}{3.6b}\label{lem:up}If $n\geq 1$ then $$\left|1 + \frac{q_{n-1}}{q_nz_n}\right| < \frac{4\eps^2\mu^2}{(1-\eps^2)^2}.$$\end{customlem}

\begin{proof}First we use Theorem \ref{thm:growq} with $n'=n-1$ to get $$|q_n|>\frac{(1-\eps^2)^2|q_{n-1}z_{n-1}|}{4\eps\mu^2}\geq \frac{(1-\eps^2)^2|q_{n-1}|}{4\eps^2\mu},$$ which bounds $|q_{n-1}/q_n|$ from above. Along with $1/|z_n|<\eps$ and the triangular inequality, this shows $$\left|1 + \frac{q_{n-1}}{q_nz_n}\right|< 1+\frac{4\eps^3\mu^2}{(1-\eps^2)^2}.$$ The right-hand side above is less than the stated bound precisely when $\eps\mu$ exceeds $(1-\eps^2)/2\sqrt{1-\eps}$, which attains a maximum of $2\sqrt{6}/9$ at $\eps=1/3$. By Lemma \ref{lem:er}, $\eps\mu\geq 2/3 > 2\sqrt{6}/9$.\end{proof}

\begin{customthm}{3.7b}\label{thm:mainb}If $n\geq 1$ then $|q_nz-p_n|$ is greater than $$i)\;\;\frac{(1-\eps^2)^2|b_n|}{4\eps^2\mu^2|q_nz_n|},\hspace{1cm}ii)\;\;\frac{(1-\eps^2)^2|b_{n+1}|}{4\eps^2\mu^2|q_{n+1}|},\hspace{0.5cm}\text{and}\hspace{0.5cm}iii)\;\;\frac{(1-\eps^2)^2|b_{n+1}|}{8\eps^2\mu^2|a_{n+1}q_n|}.$$\end{customthm}

\begin{proof}Lemma \ref{lem:up} combines with identities (\ref{eq:7}) and (\ref{eq:8}) to prove $i)$  and $ii)$ directly, just as in the proof of Theorem \ref{thm:main}. 

For $iii)$, $|b_{n+1}z_n-a_{n+1}|\leq\eps|b_n|$ implies $|b_{n+1}z_n|\leq |a_{n+1}|+\eps|b_n|$. Dividing both sides of this last inequality by $|a_{n+1}b_n|$ gives $$\left|\frac{b_{n+1}z_n}{a_{n+1}b_n}\right|\leq \frac{1}{|b_n|}+\frac{\eps}{|a_{n+1}|}\leq 2.$$ Now scale both ends of the inequality above by $|b_n/2z_n|$ to get the lower bound for $|b_n/z_n|$ that turns $i)$ into $iii)$.\end{proof}

\begin{corollary}\label{cor:bad}An input $z$ is badly approximable if and only $(a_n/b_n)_n$ is bounded.\end{corollary}

\begin{proof} If $z$ is badly approximable then $(a_n)_n$ (and thus $(a_n/b_n)_n$) is bounded by Theorem \ref{thm:main} $iii)$. If $z$ is well approximable then sufficiently good approximations appear as convergents by Theorem \ref{thm:best}, implying $(a_n/b_n)_n$ is unbounded by Theorem \ref{thm:mainb} $iii)$.\end{proof}

Finally, we prove the existence of periodic expansions of quadratic, irrational inputs.

\begin{theorem}\label{thm:quad}The set $\{z_n\}_n$ is finite if and only if $[K(z):K] \leq 2$. In particular, $z$ has a continued fraction expansion in which the sequence of pairs $(a_n,b_n)_n$ is eventually periodic and infinite if and only if $[K(z):K]=2$.\end{theorem}

\begin{proof}If $\{z_n\}_n$ is finite and $(z_n)_n$ is not then there are distinct $n,n'\in\bN$ with $M^{-1}_n(z) = z_n = z_{n'} = M^{-1}_{n'}(z)$. By Corollary \ref{cor:nzero}, $M_n$ cannot be a scaled copy of $M_{n'}$. Thus $M_{n'}M^{-1}_n(z) = z$ shows that $z$ satisfies a quadratic (irreducible by Corollary \ref{cor:term}) polynomial with coefficients in $K$.

For the converse, suppose $[K(z):K] = 2$. Let $x$ denote the discriminant of an integral, quadratic polynomial of which $z$ is a root, and use the quadratic formula to write $z=(w + x)/y$ such that $w,y,(x^2-w^2)/y\in\cO$. Let us also get an expression for $z_n$ in terms of $w$, $x$ and $y$. The last equality below comes from rationalizing the denominator with respect to the quadratic irrational $x$. That is, we multiply numerator and denominator by the conjugate of the denominator. $$z_n=M_n^{-1}(z)=\frac{q_{n-1}z-p_{n-1}}{p_n-q_nz} = \frac{(q_{n-1}w-p_{n-1}y)+q_{n-1}x}{(p_ny-q_nw)-q_nx}=$$ \begin{equation}\frac{\big(q_{n-1}q_n(x^2-w^2)/y - p_{n-1}p_ny+w(p_nq_{n-1}+q_np_{n-1})\big)+(-1)^nb_nx}{\big((p_ny-q_nw)^2-q_n^2x^2\big)/y}.\label{eq:103}\end{equation} Let $x_n = (-1)^nb_nx$, giving $x_n^2\in\cO$. Define $w_n$ and $y_n$ using the expression above so that $(w_n+x_n)/y_n=z_n$. As $y$ divides $x^2-w^2$, $w_n$ and $y_n$ are seen from (\ref{eq:103}) to be integers. 

Now we find recursive expressions for $w_n$, $x_n$, and $y_n$ using the recursive formula for $z_n$ in Proposition \ref{prop:rels}. The final equality is again rationalizing the denominator. $$\frac{w_n+x_n}{y_n}=z_n=\frac{b_{n-1}}{b_nz_{n-1} - a_n} = \frac{b_{n-1}y_{n-1}}{b_n(w_{n-1}+x_{n-1})-a_ny_{n-1}}=$$ \begin{equation}\frac{\big((b_nw_{n-1}-a_ny_{n-1})/b_{n-1}\big)-b_nx_{n-1}/b_{n-1}}{\big((b_nw_{n-1}-a_ny_{n-1})^2-b_n^2x_{n-1}^2\big)/b_{n-1}^2y_{n-1}}.\label{eq:104}\end{equation} Note that $-b_nx_{n-1}/b_{n-1} = -b_n\big((-1)^{n-1}b_{n-1}x\big)/b_{n-1} = (-1)^nb_nx$, thus matching a pair of terms from (\ref{eq:103}) and (\ref{eq:104}). But $\{1,x\}$ is a basis for the field extension $K(z)/K$. So because the remaining terms in (\ref{eq:103}) and (\ref{eq:104}) belong to $K$, there is only one way they could match up: $$w_n=\frac{b_nw_{n-1}-a_ny_{n-1}}{b_{n-1}}\hspace{0.5cm}\text{and}\hspace{0.5cm}y_n=\frac{(b_nw_{n-1}-a_ny_{n-1})^2-b_n^2x_{n-1}^2}{b_{n-1}^2y_{n-1}}.$$ These recursive formulas for $w_n$ and $x_n$ give the second equality below: $$\frac{1}{\eps}\left|\frac{y_n}{y_{n-1}}\right|\leq\left|\frac{y_nz_n}{y_{n-1}}\right| = \left|\frac{w_n+x_n}{y_{n-1}}\right|=\left|\frac{b_nz_{n-1}-a_n}{b_{n-1}}+\frac{(-1)^n2b_nx}{y_{n-1}}\right|\leq\eps+\frac{2\mu|x|}{|y_{n-1}|}.$$ Thus $|y_n|\leq\eps^2|y_{n-1}|+2\eps\mu|x|$, implying $(y_n)_n$ is a bounded sequence. Therefore $$|w_n| = \left|\frac{b_nw_{n-1}-a_ny_{n-1}}{b_{n-1}}\right|=\left|\frac{y_{n-1}(b_nz_{n-1}-a_n)}{b_{n-1}}+(-1)^nb_nx\right|\leq \eps|y_{n-1}|+\mu|x|$$ shows that $(w_n)_n$ is also bounded. Since $w_n$, $x_n^2$, and $y_n$ are all bounded integers, $\{z_n\}_n=\{(w_n+x_n)/y_n\}_n$ is finite.

To see why the final periodicity claim follows, fix an expansion of a quadratic irrational $z$. By finiteness of $B$ and $\{z_n\}_n$, there are indices $n'> n$ with $z_{n'}=z_n$, $b_{n'}=b_n$, and $M_{n'}\equiv M_n\,\text{mod}\,b_n$. For any matrix $S$, if either of $M_nS$ or $M_{n'}S$ has integer entries then $b_nS=\det M_nS=\det M_{n'}S$ does too, implying both $M_nS$ and $M_{n'}S$ have integer entries. This shows that $(M_nS)_{\ell}$ is $\eps$-reduced if and only if $(M_{n'}S)_{\ell}$ is $\eps$-reduced. Thus we may take $a_{k'}=a_k$ and $b_{k'}=b_k$ for all $k'>n'$, where $k'\equiv k\,\text{mod}\,(n'-n)$ for $n<k\leq n'$.\end{proof}

We remark that aside from being overly complicated, the proof of Theorem \ref{thm:quad} applies to the continued fractions produced when Algorithm \ref{alg:1} is executed over $\bZ$. The author is not aware of such a perspective (absent of a fixed convention for selecting among multiple coefficients) in the literature. Even with $B=\{1\}$, there are overlapping discs (intervals in this case) that allow for an infinite number of periodic continued fraction expansions, all of which converge to the given quadratic, irrational input. By taking $B\neq\{1\}$, our algorithm finds the additional use over $\bZ$ of producing even more such expansions. It is possible, however, that these expansions are already obtainable by altering nearest integer coefficients with the processes of singularization and expansion \cite{kraaikamp}.

\begin{figure}[ht]
   \begin{minipage}{0.49\textwidth}
     \centering
     \begin{overpic}[height=7.5cm, clip, trim=2cm 0cm 2cm 0cm,unit=1mm]{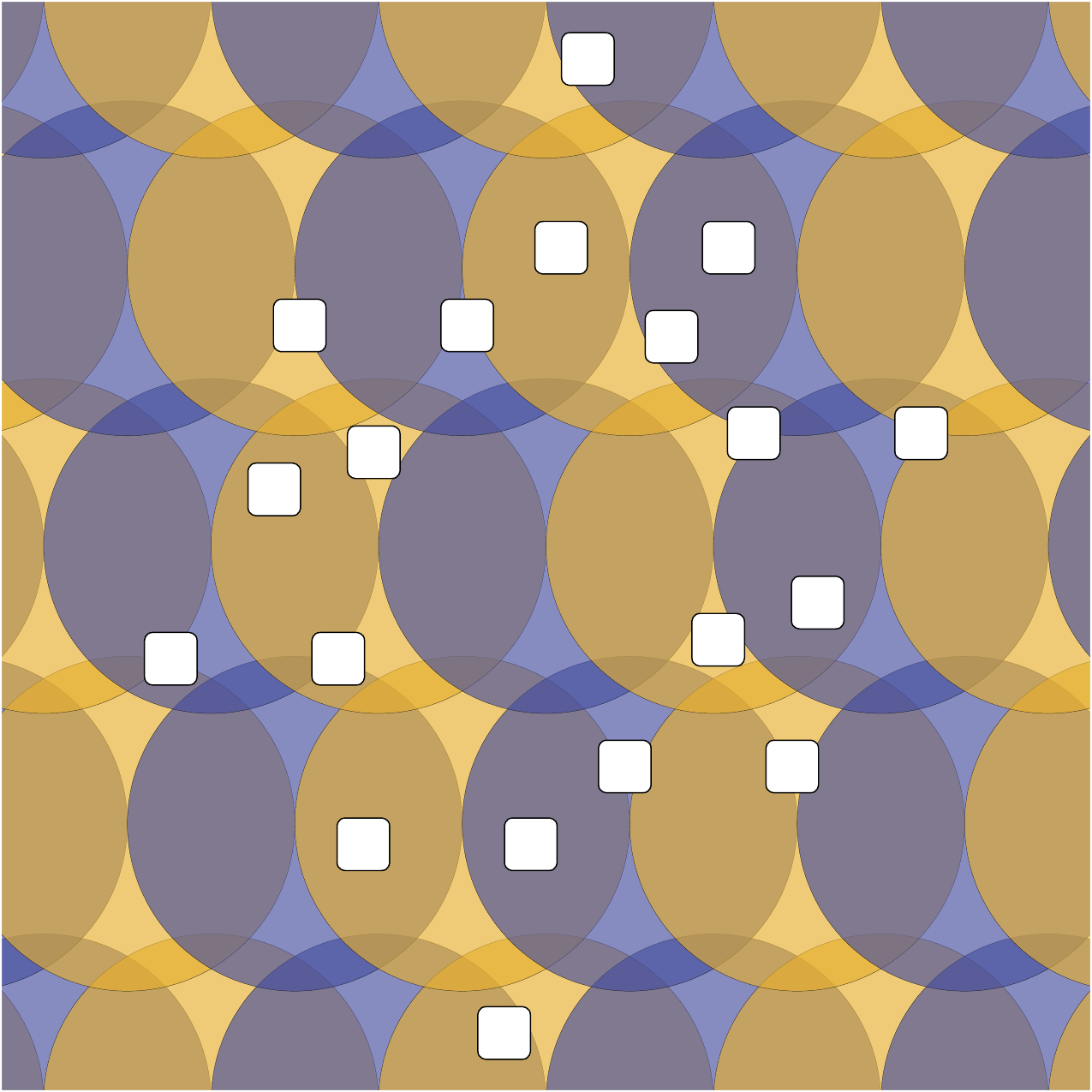}
    \put(38.4,50.8){\small$0$}
    \put(18,42.9){\small$1$}
    \put(41.15,30){\small\minus$1$}
    \put(42.3,56.9){\small$7$}
    \put(16.6,16){\small\minus$7$}
    \put(28.7,16){\small$2$}
    \put(30.25,57){\small\minus$2$}
    \put(48.4,32.6){\small$3$}
    \put(10.5,40.4){\small\minus$3$}
    \put(4,28.7){\small$8$}
    \put(55.0,44.25){\small\minus$8$}
    \put(44,44.2){\small$4$}
    \put(14.9,28.7){\small\minus$4$}
    \put(12.9,51.6){\small$9$}
    \put(46.05,21.3){\small\minus$9$}
    \put(24.4,51.6){\small$5$}
    \put(34.65,21.3){\small\minus$5$}
    \put(32.65,69.85){\small$6$}
    \put(26.35,3){\small\minus$6$}
    \end{overpic}
   \end{minipage}\hfill
   \begin{minipage}{0.468\textwidth}
     \centering
     \vspace{-0.21cm}
     \begin{tikzpicture}[scale = 0.82]
    \draw(3,2.5,0) node (0) {0};
    \draw(5,4.4,0) node (1) {7};
    \draw(6.231,2.35,0) node (2) {1};
    \draw(3.2,7.6,0) node (4) {3};
    \draw(2.6,5.7,0) node (3) {8};
    \draw(6.231,5.85,0) node (5) {2};
    \draw(0.169,2.35,0) node (6) {5};
    \draw(1.4,3.8,0) node (7) {9};
    \draw(0.169,5.85,0) node (8) {4};
    \draw(3.2,0.6,0) node (9) {6};
    \draw [bend right=20,-{Stealth[round,scale=1.5]},daisy] (2) to (5);
    \draw [bend right=20,-{Stealth[round,scale=1.5]},pacific] (9) to (2);
    \draw [bend right=20,-{Stealth[round,scale=1.5]},pacific] (0) to (2);
    \draw [bend right=20,-{Stealth[round,scale=1.5]},daisy] (0) to (1);
    \draw [bend right=20,-{Stealth[round,scale=1.5]},daisy] (9) to (1);
    \draw [bend right=20,-{Stealth[round,scale=1.5]},pacific] (1) to (3);
    \draw [bend right=20,-{Stealth[round,scale=1.5]},pacific] (3) to (6);
    \draw [bend right=20,-{Stealth[round,scale=1.5]},pacific] (6) to (9);
    \draw [bend right=20,-{Stealth[round,scale=1.5]},pacific] (4) to (7);
    \draw [bend right=20,-{Stealth[round,scale=1.5]},daisy] (1) to (4);
    \draw [bend right=20,-{Stealth[round,scale=1.5]},daisy] (4) to (8);
    \draw [bend right=20,-{Stealth[round,scale=1.5]},daisy] (7) to (9);
    \draw [bend right=20,-{Stealth[round,scale=1.5]},daisy, dashed] (5) to (3);
    \draw [bend right=20,-{Stealth[round,scale=1.5]}, dashed,daisy] (8) to (6);
    \draw [bend right=20,-{Stealth[round,scale=1.5]}, dashed,pacific] (5) to (4);
    \end{tikzpicture}
    \captionsetup{width=1\textwidth,aboveskip=0.6\baselineskip}
    \caption{Left: $\{z_n\}_n$ for input $z_0=(3+5i)/4$ and $|\Delta|=11$. Right: result from using the disc of indicated color.}\label{fig:5}
   \end{minipage}
\end{figure}

Figure \ref{fig:5} shows $\{z_n\}_n$ for $z = (3+5i)/4$ using $B=\{1\}$ in $\bQ(\sqrt{-11})$. The covering is centered at $0\in\bC$, and $z$ is labeled ``0." As it lies in both the yellow disc centered at $(1+\sqrt{-11})/2$ and the blue disc centered at $(3+\sqrt{-11})/2$, there are two possibilities for $a_1$. The resulting values of $z_1$ are indicated by the yellow and blue arrows to 1 and 7 in the diagram, and are labeled ``1" and ``7" in the image. In particular, numbers in the figure need not correspond to the stage number at which a point might appear.

It is not generally the case that if $z_n$ appears in a continued fraction expansion then $-z_n$ will also appear, but it does happen in this example for every point except $z_0=z$. This fact has been used to cut the number of nodes needed for our diagram in half. Dashed arrows indicate a sign switch. For example, consider the point labeled ``$-4$," which is only contained in the disc centered on $-1$. Its image under $z\mapsto 1/(z-(-1))$ is ``$5$," not ``$-5$." Similarly, ``2" is mapped by $z\mapsto1/(z-(1-\sqrt{-11})/2)$ to ``$-3$."

\section{Admissible parameters}\label{sec:4}

The coverings provided by Definition \ref{def:admit} are indexed by ideals rather than the matrices used in line 5. Let us check that our notion of admissibility is sufficient for Algorithm \ref{alg:1} to function, meaning $a$ and $b$ satisfying lines 4 and 5 always exist.

\begin{lemma}\label{lem:ideal}For $M,M'\in\emph{GL}_2(K)$, if $M$ is integral then $(MM')_{\ell}\supseteq\det M(M')_{\ell}$.\end{lemma}

\begin{proof}The adjugate of $M$, $\text{adj}\,M=M^{-1}\det M$, is also integral. So $\det M(M')_{\ell}=((\text{adj}\,M)MM')_{\ell}\subseteq(MM')_{\ell}$.\end{proof}

\begin{proposition}\label{prop:works}The \emph{\textbf{if}} condition in line 5 of Algorithm \ref{alg:1} is satisfied at least once every outer \emph{\textbf{for}} loop iteration.\end{proposition}

\begin{proof}Let $M$ with $\det M=b'\in B$ be a matrix that occurs in the execution of Algorithm \ref{alg:1}. Since $M$ is either the identity or the matrix product from line 5 in the previous outer \textbf{for} loop iteration, $(M)_{\ell}$ is $\eps$-reduced. Denote this ideal $\fb$, as that is the role it will play in Definition \ref{def:admit}. 

Fix any $a'$ in the fractional ideal $\fb^{-1}$ that makes $(MS(a',1))_{\ell}=b'\fb^{-1}$. According to Definition \ref{def:admit}, for any $z\in\bC$ there exist $a\in K$ and $b\in B$ for which $(a\fb,b\fb^{-1})$ is integral and $\eps$-reduced and $(M^{-1}(z)-a')/b'\in D(a/b,\eps/|b|)$. This disc containment rearranges to $ab'+a'b\in D(bM^{-1}(z),\eps|b'|)$ as in line 4. Note that $b',b\in\fb$ implies $ab'+a'b\in\cO$. So our proposed choice of coefficient pair is $ab'+a'b$ and $b$. It remains only to check that $(MS((ab'+a'b)/b',b/b'))_{\ell}$ is $\eps$-reduced. We claim that this ideal is exactly $(a\fb,b\fb^{-1})$, which would complete the proof. 

Let $\fp$ be an unramified prime in the ideal class of $\fb$ that does not divide $a\overline{b}\fb$, and let $k\in K$ be a generator for $\fp\fb^{-1}$. Consider the product \begin{equation}M\begin{bmatrix}1 & a' \\ 0 & 1\end{bmatrix}\begin{bmatrix}k & 0 \\ 0 & \|\fp\|/kb'\end{bmatrix}\begin{bmatrix}\|\fp\|/kb' & 0 \\ 0 & k\end{bmatrix}S(ab',b) = \frac{\|\fp\|}{b'}MS(ab'+a'b,b).\label{eq:102}\end{equation} We chose $a'$ so that the product of the first two matrices has right-column ideal $b'\fb^{-1}$ and left-column ideal $\fb$. In particular, the product of the first three matrices, call this $M'$, has determinant $\|\fp\|$, with $(M')_{\ell}=\fp$ and $(M')_r=\overline{\fp}$. The product of the last two matrices on the left-hand side has top-left entry $a\|\fp\|/k$, which generates $a\fb\overline{\fp}$, and bottom-left entry $bk$, which generates $b\fp\fb^{-1}$. Thus the left-column ideal of the product of all five matrices is contained in $((M')_{\ell}\hspace{0.02cm}a\fb\overline{\fp},(M')_rb\fp\fb^{-1})=\|\fp\|(a\fb,b\fb^{-1})$. On the other hand, since $M'$ is integral, Lemma \ref{lem:ideal} says that the left-column ideal of the overall product contains $\det M'(a\|\fp\|/k,bk) = \|\fp\|(a\fb\overline{\fp},b\fp\fb^{-1})$. Since $\fp$ is coprime to $a\fb\overline{\fp}$ and $\overline{\fp}$ is coprime to $b\fp\fb^{-1}$, we have $\|\fp\|(a\fb\overline{\fp},b\fp\fb^{-1})=\|\fp\|(a\fb,b\fb^{-1})$, which must then equal the left-column ideal of the overall product. Comparing to the right-hand side of (\ref{eq:102}) proves our claim.\end{proof}

\subsection{Generic admissible sets}\label{ss:admit} The following result confirms the main assertion of this paper: Algorithm \ref{alg:1} works in any imaginary quadratic field. The admissible set $\big\{1,2,...,\big\lfloor\sqrt{|\Delta|}\big\rfloor\big\}$ is highlighted because it makes decent constants in the previous section's results. A few such constants can be seen in Theorems \hyperref[thm:1]{1.2}, \hyperref[thm:2]{1.3}, and \hyperref[thm:3]{1.4}.

\begin{theorem}\label{thm:ints}If $\mu\in\bZ$ with $\mu \geq\big\lfloor\!\sqrt{|\Delta|/3}\big\rfloor$, then $\{1,2,...,\mu\}$ is admissible with \begin{equation}\eps = \frac{1}{2}\sqrt{1+\frac{|\Delta|}{(\mu+1)^2}}\label{eq:106}\end{equation} provided $2\eps^2\mu < \sqrt{|\Delta|}$. In particular, $\big\{1,2,...,\big\lfloor\sqrt{|\Delta|}\big\rfloor\big\}$ is admissible with $ 1/\sqrt{2}$.\end{theorem}

\begin{proof}Fix $z\in\bC$ and $\fb\subseteq\cO$ with $\fb\cap B\neq \emptyset$. We will show that $z$ is contained in Definition \ref{def:admit}'s union.

Let $b_1\in \fb\cap B$ be minimal, implying the integral ideal $b_1\fb^{-1}$ has no nontrivial rational divisors. By Dirichlet's approximation lemma, the real number $2b_1^2\Im(z)/\sqrt{|\Delta|}$ admits a rational approximation $a_2/b_2$ for coprime $a_2,b_2\in\bZ$ with $1\leq b_2\leq\lfloor \mu/b_1\rfloor$ and \begin{equation}\left|\frac{2b_1^2\Im(z)}{\sqrt{|\Delta|}}-\frac{a_2}{b_2}\right| \leq \frac{1}{b_2(\lfloor\mu/b_1\rfloor + 1)}.\label{eq:107}\end{equation} Scaling both sides above by $b_2\sqrt{|\Delta|}/2b_1$ is the first inequality below; the second inequality uses the fact that $\mu-b_1\lfloor\mu/b_1\rfloor$, which is a rational integer less than $b_1$, is at most $b_1-1$: 

$$\left|b_1b_2\Im(z)-\frac{a_2\sqrt{|\Delta|}}{2b_1}\right| \leq \frac{\sqrt{|\Delta|}}{2b_1(\lfloor\mu/b_1\rfloor + 1)}\leq \frac{\sqrt{|\Delta|}}{2(\mu+1)}.$$

Since $b_1\fb^{-1}$ has no nontrivial rational divisors, there is a congruence class modulo $2\|b_1\fb^{-1}\|$ (in $\bZ$) whose elements, call one $a_1$, satisfy $(a_1+a_2\sqrt{|\Delta|})/2\in b_1\fb^{-1}$. Choose such an $a_1$ nearest to $2b_1^2b_2\Re(z)$. Then \begin{equation}\left|b_1b_2z-\frac{a_1+a_2\sqrt{\Delta}}{2b_1}\right|\leq\sqrt{\frac{\|b_1\fb^{-1}\|^2}{4b_1^2}+\frac{|\Delta|}{4(\mu+1)^2}}\leq \eps.\label{eq:105}\end{equation} Thus $z\in D(a/b,\eps/|b|)$, where $a = (a_1 + a_2\sqrt{\Delta})/2b_1$ and $b = b_1b_2\in B$.

Assume by way of contradiction that $(a\fb,b\fb^{-1})$ is not $\eps$-reduced. From Definition \ref{def:reduce}, this means $k(a\fb,b\fb^{-1})\subseteq\cO$ for some $k\in K$ with $0 < |k|\leq \eps^2$. Then the assumption $2\eps^2\mu < \sqrt{|\Delta|}$ gives $|k|b\leq \eps^2\mu < \sqrt{|\Delta|}/2$. Thus $kb\in \bZ$ because a non-rational integer must have magnitude at least $\sqrt{|\Delta|}/2$. By writing $k\in\bQ$ in lowest terms, we see that $k(a\fb,b\fb^{-1})\subseteq\cO$ implies $(a\fb,b\fb^{-1})$ has a nontrivial rational divisor. Finally, $a\fb = (\fb/b_1)(a_1+a_2\sqrt{|\Delta|})/2\ni (a_1+a_2\sqrt{|\Delta|})/2$ shows that a rational divisor of $a\fb$ must divide $a_2$, and $b\fb^{-1} = (b_1\fb^{-1})b_2$ shows that a rational divisor of $b\fb^{-1}$ must divide $b_2$ (recalling that $b_1\fb^{-1}$ has no nontrivial rational divisors). But $a_2$ and $b_2$ are coprime.\end{proof}

The second paragraphs in the proofs of Proposition \ref{prop:works} and Theorem \ref{thm:ints} are algorithmic in nature. Combining them gives a fast method for finding coefficients.

\begin{algorithm}[H]\caption{Subroutine for Algorithm \ref{alg:1} to find coefficients ($a_n$ and $b_n$) under the parameters $B$ and $\eps$ defined in Theorem \ref{thm:ints}.}\label{alg:2}
\begin{flushleft}
\hspace*{\algorithmicindent}\textbf{input:} $b'$, $B$, $M$, $z$ used in line 4 of Algorithm \ref{alg:1}; also let $\fb$ denote $(M)_{\ell}$\\
\hspace*{\algorithmicindent}\textbf{output:} coefficient pair satisfying lines 4 and 5 of Algorithm \ref{alg:1}. 
\end{flushleft}
\begin{algorithmic}[1]
    \State $a'\gets$ number with $(MS(a',1))_{\ell}=b'\fb^{-1}$
    \State $b_1\gets$ minimum element of $\fb\cap B$\Top{$b'=\det M$ makes $\fb\cap B\neq\emptyset$}
    \State $z\gets b_1^2(M^{-1}(z)-a')/b'$
    \State $a_2/b_2\gets$ last convergent of $2\Im(z)/\sqrt{|\Delta|}$\Longtop{from classic floor function con-}
    \Statex that makes $b=b_1b_2\in B$\Bottom{tinued fractions over $\bZ$}
    \State $a_1\gets$ nearest rational integer to $2b_2\Re(z)$
    \Statex such that $(a_1+a_2\sqrt{|\Delta|})/2b_1\in\fb^{-1}$
    \State $a\gets (a_1+a_2\sqrt{|\Delta|})/2b_1$
    \State \Return $ab'+a'b$, $b$
\end{algorithmic}
\end{algorithm}

Let us collect notation that has been used thus far. Both Algorithm \ref{alg:2} and the proof of Proposition \ref{prop:works} begin with $b'=\det M$ and $\fb=(M)_{\ell}$ (which is $(M_{n-1})_{\ell}=(p_{n-1},q_{n-1})$ if Algorithm \ref{alg:2} is returning $a_n$ and $b_n$). Both also use $ab'+a'b$ and $b$ to denote the coefficient pair returned to Algorithm \ref{alg:1}. Then Algorithm \ref{alg:2} and the proof of Theorem \ref{thm:ints} further break these variable down into $a=(a_1+a_2\sqrt{|\Delta|})/2b_1$ and $b=b_1b_2$. This notation is used throughout the remainder of the subsection.

\begin{corollary}\label{cor:works}Under parameters $B$ and $\eps$ as defined in Theorem \ref{thm:ints}, the coefficients produced by Algorithm \ref{alg:2} satisfy lines 4 and 5 of Algorithm \ref{alg:1}.\end{corollary}

\begin{proof}Algorithm \ref{alg:2} is pseudocode for the second paragraph in the proof of Theorem \ref{thm:ints}. The only difference is that ``$z$" in Theorem \ref{thm:ints} is $(M^{-1}(z)-a')/b'$ in Algorithm \ref{alg:2}. This substitution for $z$ makes (\ref{eq:105}) equivalent to $ab'+a'b\in D(bM^{-1}(z),\eps|b'|)$ as required by line 4 of Algorithm \ref{alg:1}. 

Regarding line 4 of Algorithm \ref{alg:2}, floor function continued fractions return \textit{all} $a_2/b_2$  for which $|b_2(2\Im(z)/\sqrt{|\Delta|})-a_2|$ is minimal without increasing $b_2$. So because there exists an approximation satisfying (\ref{eq:107}), line 4 of Algorithm \ref{alg:2} finds it.

In the proof of Proposition \ref{prop:works} we saw that $(MS((ab'+a'b)/b',b/b'))_{\ell}=(a\fb,b\fb^{-1})$. The third paragraph of the proof of Theorem \ref{thm:ints} shows that $(a\fb,b\fb^{-1})$ is $\eps$-reduced. Thus the coefficients returned by Algorithm \ref{alg:2} satisfy line 5 of Algorithm \ref{alg:1}.\end{proof}

As discussed in Section \ref{sec:1}, two advantages to using Algorithm \ref{alg:2} are speed and control over the divisors of $(p_{n-1},q_{n-1},p_n,q_n)$.

\begin{theorem}\label{thm:coprime}If $a_{n-1}$, $b_{n-1}$, $a_n$, and $b_n$ are found using Algorithm \ref{alg:2}, then only ramified, non-rational primes divide $(p_{n-1},q_{n-1},p_n,q_n)$.\end{theorem}

\begin{proof}The last paragraph of the proof of Theorem \ref{thm:ints} shows that $(M_{n-1})_{\ell}=\fb$ has no nontrivial rational divisors if $a_{n-1}$ and $b_{n-1}$ are chosen using Algorithm \ref{alg:2}. Therefore $b_1$ defined in line 2 is just $\|\fb\|$. We have seen that $(p_n,q_n)=(a\fb,b\fb^{-1})$, so the goal is to verify that $(\fb,a\fb,b\fb^{-1})$ has no split prime divisors. 

Suppose $\fp$ divides $\fb$ but $\overline{\fp}$ does not. Then $\fp$ divides $b\fb^{-1} = b_2\overline{\fb}$ if and only if it divides $b_2$. Similarly, $\fp$ divides $a\fb = (a_1+a_2\sqrt{\Delta})/2\overline{\fb}$ if and only if it divides $(a_1+a_2\sqrt{\Delta})/2$. But this is true if and only if $\|\fp\|$ divides $a_2$, because the choice of $a_1$ in line 5 gives $(a_1+a_2\sqrt{\Delta})/2\in\overline{\fb}\subseteq\overline{\fp}$. Thus $\text{gcd}(a_2,b_2)=1$ implies $\fp$ does not divide $(a\fb,b\fb^{-1})$ as desired.\end{proof}

We gauge the efficiency of Algorithms \ref{alg:1} and \ref{alg:2} by time required to find $p,q\in\cO$ with $q\neq 0$ satisfying $|qz-p|<1/\delta$ for some approximation quality goal $\delta$. To make input length well-defined, the usual $z\in\bC$ is replaced with $z\in K$. Then $\log|wxy\Delta|$ can be taken as the input length of $z=(w+x\sqrt{\Delta})/y$, where $w,x,y\in\bZ$ and $\text{gcd}(w,x,y)=1$. 

\begin{theorem}\label{thm:timebest}Let $z$ have input length $\ell$ and let $\delta \geq 2$. By using Algorithm \ref{alg:2} to compute coefficients, Algorithm \ref{alg:1} can be executed in $O(\log|\Delta|\log_{\eps}\!\delta)$ operations on integers of length $O(\ell+\log \delta|\Delta|)$ to find $p,q\in \cO$ with $q\neq 0$ and $|qz-p|\leq 1/\delta$.\end{theorem}

\begin{proof}To achieve $|qz-p|\leq 1/\delta$, at most $\lceil \log _{1/\eps}\delta\rceil$ outer \textbf{for} loop iterations are needed by Proposition \ref{prop:mono}. Let us determine the cost of Algorithm \ref{alg:2}.

Fix inputs for Algorithm \ref{alg:2}. Consider the four-element generating set (over $\bZ$) for $\fb=(M)_{\ell}$ consisting of the products of left-column entries of $M$ with a $\bZ$-basis for $\cO$. Reduce them $\text{mod}\,b'$. By computing greatest common divisors among the rational integers defining the real parts and imaginary parts of these four generators, it is straightforward to reduce our set to a $\bZ$-basis for $\fb$ in $O(\log b')$ operations. Once we have this basis, line 1 is a matter of solving a system of two inhomogeneous congruences, which also requires $O(\log b')$ operations for the Euclidean algorithm. 

As we are assuming Algorithm \ref{alg:2} was used on the previous \textbf{for} loop iteration, $\fb$ has no nontrivial rational divisors. Thus $b_1=\|\fb\|$, the determinant of the two-by-two matrix with columns coming from our two-element $\bZ$-basis for $\fb$. So lines 2 and 3 both take $O(1)$ operations. 

Using classical continued fractions for line 4 requires $O(\log(\mu/b_1))$ operations.

Line 5 requires computing the appropriate congruence class for $a_1\,\text{mod}\,2\|\fb\|$, since $(a_1+a_2\sqrt{|\Delta|})/2\in b_1\fb^{-1} = \overline{\fb}$. (The factor of $2$ in the modulus comes from the general constraint $a_1\equiv a_2\Delta\,\text{mod}\,2$.) This requires $O(\log\|\fb\|)$ operations.

So at worst, a line in Algorithm \ref{alg:2} requires $O(\log\mu)$ operations. The inequality $2\eps^2\mu<\sqrt{|\Delta|}$ from Theorem \ref{thm:ints} along with the definition of $\eps$ in (\ref{eq:106}) imply $\mu=O(\sqrt{|\Delta|})$, giving the desired overall bound on operations of $O(\log|\Delta|\log_{\eps}\!\delta)$. 

We turn to the bound on integer lengths. Let $n$ be the first index for which $|q_nz-p_n|\leq 1/\delta$. For $n'\leq n$, Theorem \ref{thm:main} shows $|a_{n'}|= O(\delta\mu^2)$, $|q_{n'}|=O(\delta\mu)$, and $|z_{n'}|=O(\delta\mu)$ (except possibly $z_0$). Using $|q_{n'}z-p_{n'}|< 1$ shows that $|p_{n'}| = O(\delta\mu|z|)$. And finally, $b_{n'}\leq \mu$. Computations involve a few of these variables within each stage. Since the input length of an integer is determined by the logarithm of its magnitude, $\mu=O(\sqrt{|\Delta|})$ completes the proof.\end{proof} 

\subsection{Precomputing admissible sets}\label{ss:precomp}Table \ref{table:2} shows admissible sets for $|\Delta| <50$ with their minimal $\eps$-values. For each field we include the smallest admissible set $B$ (measured by $\mu$), as well as the next smallest set which decreases the corresponding value of $\eps$. We use $\tau$ to denote $\sqrt{\Delta}/2$ or $(1+\sqrt{\Delta})/2$ according to the parity of $\Delta$.

\setlength\tabcolsep{0.18cm}
\renewcommand{\arraystretch}{1.2}

\begin{table}[ht]
\setlength{\abovecaptionskip}{-0.4\baselineskip}
\setlength{\belowcaptionskip}{0.4\baselineskip}
\noindent\begin{minipage}{0.46\textwidth}
\begin{table}[H]
\raggedright
\hspace{0.36cm}\begin{tabular}{|c|c|c|}
\hline\rowcolor{Gray}
$|\Delta|$ & $B$ & $\eps^2$\\ \hline

3 & \parbox{1.2cm}{\centering \vspace{0.1cm} 
$1$\\ $1,1+\tau$\\ 
\vspace{0.05cm}} & \parbox{2.25cm}{\centering \vspace{0.1cm} 
$1/3$\\ $1/4$\\ 
\vspace{0.05cm}}\\ \hline
4 & \parbox{1.35cm}{\centering \vspace{0.1cm} 
$1$\\ $1,1+\tau$\\ 
\vspace{0.05cm}} & \parbox{2.25cm}{\centering \vspace{0.1cm} 
$1/2$\\ $2-\sqrt{3}$\\ 
\vspace{0.05cm}}\\ \hline
7 & \parbox{1.35cm}{\centering \vspace{0.1cm} 
$1$\\ $1,\tau$\\ 
\vspace{0.05cm}} & \parbox{2.25cm}{\centering \vspace{0.1cm} 
$4/7$\\ $1/2$\\ 
\vspace{0.05cm}}\\ \hline
8 & \parbox{1.35cm}{\centering \vspace{0.1cm} 
$1$\\ $1,\tau$\\ 
\vspace{0.05cm}} & \parbox{2.25cm}{\centering \vspace{0.1cm} 
$3/4$\\ $3-\sqrt{6}$\\ 
\vspace{0.05cm}}\\ \hline
11 & \parbox{1.35cm}{\centering \vspace{0.1cm} 
$1$\\ $1,\tau$\\ 
\vspace{0.05cm}} & \parbox{2.25cm}{\centering \vspace{0.1cm} 
$9/11$\\ $3/4$\\ 
\vspace{0.05cm}}\\ \hline
15 & \parbox{1.35cm}{\centering \vspace{0.1cm} 
$1,2$\\ $1,2,1+\tau$\\ 
\vspace{0.05cm}} & \parbox{2.25cm}{\centering \vspace{0.1cm}
$2/3$ \\ $(6-2\sqrt{5})/3$\\
\vspace{0.05cm}}\\ \hline
19 & \parbox{1.35cm}{\centering \vspace{0.1cm} 
$1,2$\\ $1,\tau,\overline{\tau}$\\ 
\vspace{0.05cm}} & \parbox{2.25cm}{\centering \vspace{0.1cm}
$7/9$ \\ $(13-\sqrt{57})/8$\\
\vspace{0.05cm}}\\ \hline
20 & \parbox{1.35cm}{\centering \vspace{0.1cm} 
$1,2$\\ $1,2,1+\tau$\\ 
\vspace{0.05cm}} & \parbox{2.25cm}{\centering \vspace{0.1cm}
$(28-2\sqrt{115})/9$ \\ $(25-\sqrt{355})/9$\\
\vspace{0.05cm}}\\ \hline
\end{tabular}
\end{table}
\end{minipage}\begin{minipage}{0.54\textwidth}
\renewcommand{\arraystretch}{2.027}
\begin{table}[H]
\raggedleft
\begin{tabular}{|c|c|c|}
\hline
23 & \parbox{1.4cm}{\centering \vspace{0.1cm} 
$1,2$\\ $1,2,\tau,\overline{\tau}$\\ 
\vspace{0.05cm}} & \parbox{2.95cm}{\centering \vspace{0.1cm}
$8/9$ \\ $(31-\sqrt{161})/25$\\
\vspace{0.05cm}}\\ \hline
24 & \parbox{1.4cm}{\centering \vspace{0.1cm} 
$1,2$\\ $1,2,1+\tau$\\ 
\vspace{0.05cm}} & \parbox{2.95cm}{\centering \vspace{0.1cm}
$(11-6\sqrt{2})/3$ \\ $(10-\sqrt{58})/3$\\
\vspace{0.05cm}}\\ \hline
31 & \parbox{1.65cm}{\centering \vspace{0.1cm} 
$1,2,3$\\ $1,2,3,1+\tau$\\ 
\vspace{0.05cm}} & \parbox{2.95cm}{\centering \vspace{0.1cm}
$(191-3\sqrt{1209})/128$ \\ $20/31$\\
\vspace{0.05cm}}\\ \hline
35 & \parbox{1.7cm}{\centering \vspace{0.1cm} 
$1,2,3$\\ $1,\tau,\overline{\tau},1+\tau$\\ 
\vspace{0.05cm}} & \parbox{2.95cm}{\centering \vspace{0.1cm} 
$(211-3\sqrt{1505})/128$\\ $(805-5\sqrt{25585})/8$\\ 
\vspace{0.05cm}}\\ \hline
39 & \parbox{1.65cm}{\centering \vspace{0.1cm} 
$1,2,3$\\ $1,2,3,1+\tau$\\ 
\vspace{0.05cm}} & \parbox{2.95cm}{\centering \vspace{0.1cm}
$(231-3\sqrt{1833})/128$ \\ $10/13$\\
\vspace{0.05cm}}\\ \hline
40 & \parbox{1.65cm}{\centering \vspace{0.1cm} 
$1,2,3$\\ $1,2,3,2+\tau$\\ 
\vspace{0.05cm}} & \parbox{2.95cm}{\centering \vspace{0.1cm}
$7/8$ \\ $(25-\sqrt{185})/16$\\
\vspace{0.05cm}}\\ \hline
43 & \parbox{1.65cm}{\centering \vspace{0.1cm} 
$1,2,3$\\ $1,2,3,1+\tau$\\ 
\vspace{0.05cm}} & \parbox{2.95cm}{\centering \vspace{0.1cm}
$(251-3\sqrt{2193})/128$ \\ $391/477$\\
\vspace{0.05cm}}\\ \hline
47 & \parbox{1.65cm}{\centering \vspace{0.1cm} 
$1,2,3$\\ $1,2,3,1+\tau$\\ 
\vspace{0.05cm}} & \parbox{2.95cm}{\centering \vspace{0.1cm}
$(271-3\sqrt{2585})/128$ \\ $42/47$\\
\vspace{0.05cm}}\\ \hline
\end{tabular}\hspace{0.36cm}
\end{table}
\end{minipage}
\captionsetup{width=.94\linewidth}
\caption{Some small (measured by $\mu$) admissible sets with their minimal $\eps^2$.}\label{table:2}
\end{table}

The C\texttt{++} source code that produced Table \ref{table:2} is posted on the author's \href{https://www.math.ucdavis.edu/~dmartin}{website}. The algorithm takes as input a discriminant and a finite set of integers. It returns all admissible subsets of the input alongside their minimal $\eps$-values. In short, the algorithm works by enumerating ideals $\fb$ from Definition \ref{def:admit} and computing the smallest $\eps$ for which $\cup D(a/b,\eps/|b|)=\bC$. The covering property is checked by exploiting periodicity of the union modulo the fractional ideal $\fb^{-2}$, and verifying that intersections of the boundaries of two discs are contained in a third disc.

The admissible sets above which are contained in $\bZ$ have already been found by Theorem \ref{thm:ints}. For $\Delta=-4$, $-8$, $-15$, $-19$, $-23$, and $-40$, the values of $\eps$ given in (\ref{eq:106}) are optimal, matching those of Table \ref{table:2}. This does not always happen, and Figure \ref{fig:6} gives one such example. Both images show the same arrangement of discs, coming from the first stage of Algorithm \ref{alg:1} using $B=\{1,2,3\}$ in $\bQ(\sqrt{-47})$. There are discs of radius $1/b$ along every multiple of the line $\Im(z)=\sqrt{47}/2b$ for $b\in\{1,2,3\}$. Up to scaling, shifting, and reflecting, there are two other disc arrangements also produced by $B=\{1,2,3\}$. The three colors in Figure \ref{fig:6} distinguish among which arrangement would appear in stage 2 (similar to colors in Figures \ref{fig:3} and \ref{fig:4}).

The first image in Figure \ref{fig:6} shows the partition used by Algorithm \ref{alg:2} to determine $a_1$ and $b_1$ based on the location of $z$. Here radii can be scaled by $\eps=\sqrt{63}/8$ before disc boundaries and corresponding partition boundaries touch. This agrees with Theorem \ref{thm:ints}'s value of $\eps$ in (\ref{eq:106}). The second image shows the partition that minimizes $|b_1z-a_1|$. Now discs can be scaled by the smaller $$\eps=\sqrt{\frac{271-3\sqrt{2585}}{128}},$$ as Table \ref{table:2} asserts.

\begin{figure}
\centering
\includegraphics[trim=3.75cm 3.0cm 3.7cm 1.2cm,clip,height=5.8cm]{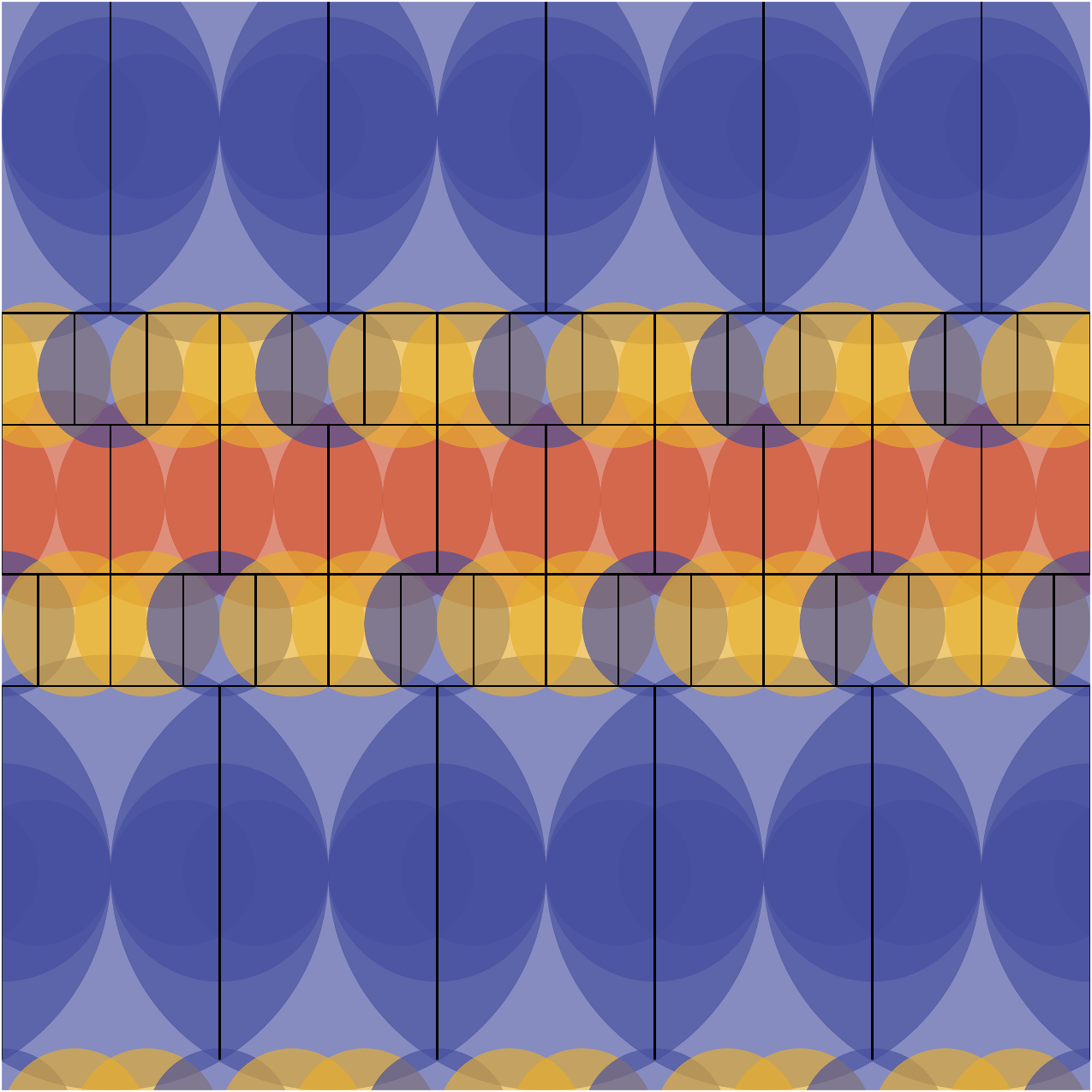}\hspace{0.5cm}
\includegraphics[trim=3.75cm 3.0cm 3.7cm 1.2cm,clip,height=5.8cm]{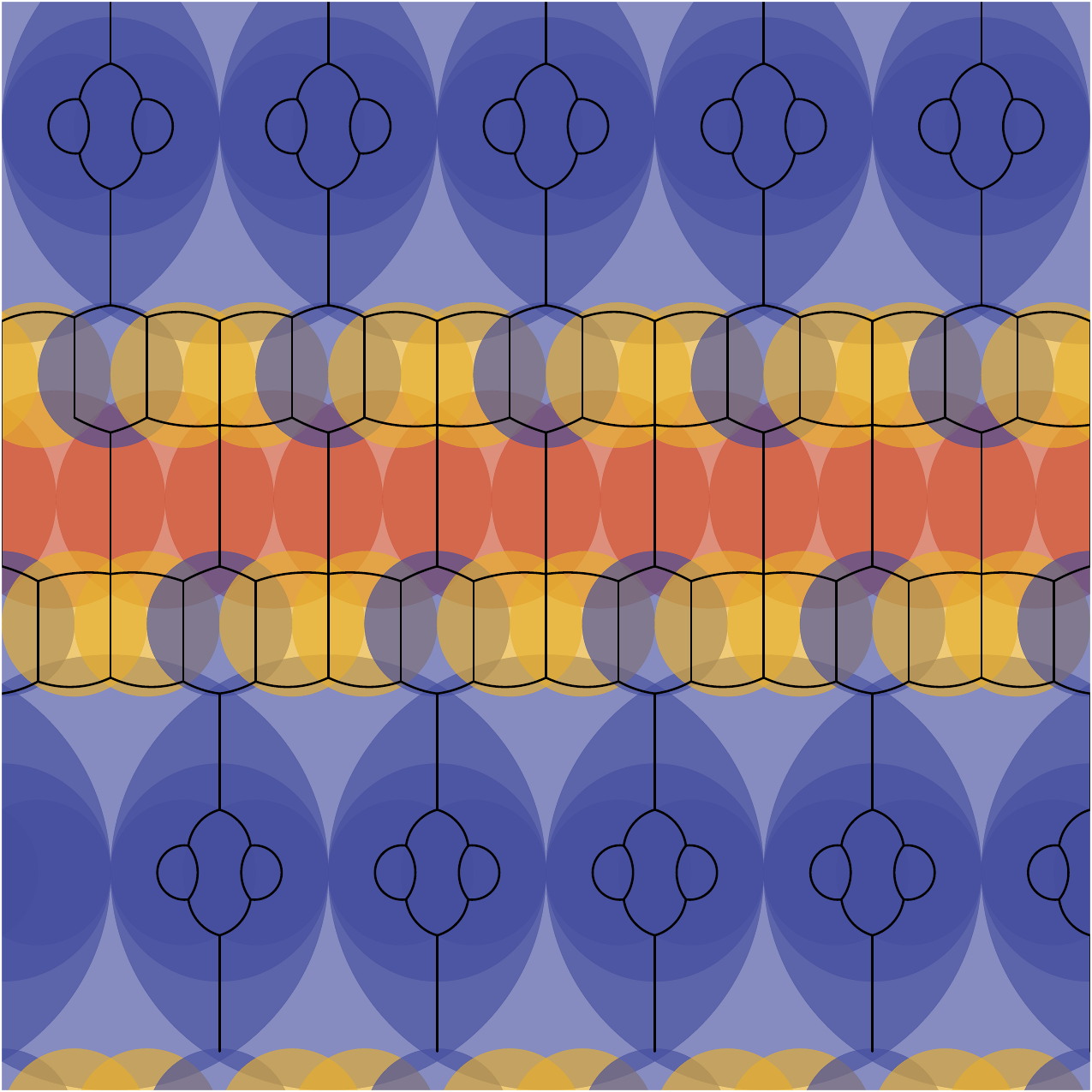}
\caption{In $\bQ(\sqrt{-47})$ using $B=\{1,2,3\}$, partitions resulting from Algorithm \ref{alg:2} (left) and minimizing $|bz-a|$ (right).}\label{fig:6}
\end{figure}

As an aside, such a partition associates to every $z\in\bC$ a disc center $a/b$. We can then ask whether a probability measure exists on $D(0,\eps)$ for which $z\mapsto (b/z-a)/b'$ is invariant and ergodic. For Hurwitz' algorithm, an invariant measure is shown to exist for $\bQ(\sqrt{-3})$ in \cite{shiokawa2}, and Nakada does the same for $\bQ(\sqrt{-1})$ in \cite{nakada}. Shiokawa also proves ergodicity results in \cite{shiokawa} for $\bQ(\sqrt{-3})$. The goal of such an investigation is to attack statistical questions, like the distribution of coefficients or the expected value of $|q_n(q_nz-p_n)|$ for $z$ uniformly distributed in $D(0,\eps)$.

\bibliographystyle{plain}
\bibliography{refs}

\end{document}